\documentclass[onefignum,onetabnum]{siamart190516}

\usepackage{lipsum}
\usepackage{amsfonts}
\usepackage{graphicx}
\usepackage{epstopdf}
\usepackage{algorithmic}
\usepackage{amsmath,ctable}
\usepackage{amssymb}
\usepackage{algorithm}
\usepackage{subfigure}
\usepackage{multirow}
\usepackage{bm}

\ifpdf
  \DeclareGraphicsExtensions{.eps,.pdf,.png,.jpg}
\else
  \DeclareGraphicsExtensions{.eps}
\fi

\makeatletter
\newcommand{\abs}[1]{\left\vert#1\right\vert}
\newcommand{\mbp}{\mathbb{P}}
\newcommand{\mbe}{\mathbb{E}}

\newcommand{\var}[1]{\mathrm{Var}(#1)}
\newcommand{\Unif}[1]{\mathbb{U}(#1)}
\newcommand{\mrd}{\mathrm{d}}
\newcommand{\mbi}{\mathbb{I}}
\newcommand{\mlmc}{\hat\theta_{\mathrm{MLMC}}}
\newcommand{\seqmlmc}{\tilde\theta_{\mathrm{MLMC}}}
\makeatother

\newsiamremark{remark}{Remark}
\newsiamremark{hypothesis}{Hypothesis}
\crefname{hypothesis}{Hypothesis}{Hypotheses}
\newsiamthm{claim}{Claim}
\newsiamthm{assum}{Assumption}

\headers{Efficient risk estimation via nested multilevel QMC simulation}
{Z. Xu, Z. He, and X. Wang}

\title{Efficient risk estimation via nested multilevel quasi-Monte Carlo simulation
\thanks{Submitted to the editors DATE.
\funding{This work was supported by the National Key R\&D Program of China (No. 2016QY02D0301) and the National Science Foundation of China (No. 72071119, 12071154), and the Fundamental Research Funds for the Central Universities (No. 2019MS106).}}}

\author{Zhenghang Xu\thanks{Department of Mathematical Sciences, Tsinghua University, Beijing 100084, People's Republic of China (\email{xzh17@mails.tsinghua.edu.cn}, \email{wangxiaoqun@mail.tsinghua.edu.cn}).}
\and Zhijian He\thanks{Corresponding author. School of Mathematics, South China University of Technology, Guangzhou 510641, People's Republic of China (\email{hezhijian@scut.edu.cn}).}
\and Xiaoqun Wang\footnotemark[2]}
\usepackage{amsopn}

\ifpdf
\hypersetup{
  pdftitle={Efficient risk estimation via nested multilevel QMC simulation},
  pdfauthor={Z. Xu, Z. He, and X. Wang}
}
\fi

\begin{document}

\maketitle

\begin{abstract}
We consider the problem of estimating the probability of a large loss from a financial portfolio, where the future loss is expressed as a conditional expectation. Since the conditional expectation is intractable in most cases, one may resort to nested simulation. To reduce the complexity of nested simulation, we present a method
that combines multilevel Monte Carlo (MLMC) and quasi-Monte Carlo (QMC). In the outer simulation, we use Monte Carlo to generate financial scenarios. In the inner simulation, we use QMC to estimate the portfolio loss in each scenario. We prove that using QMC can accelerate the convergence rates in	 both the crude nested simulation and the multilevel nested simulation. Under certain conditions, the complexity of MLMC can be reduced to $O(\epsilon^{-2}(\log \epsilon)^2)$ by incorporating QMC. On the other hand, we find that MLMC encounters catastrophic coupling problem due to the existence of indicator functions. To remedy this, we propose a smoothed MLMC method which uses logistic sigmoid functions to approximate indicator functions. Numerical results show that the optimal complexity $O(\epsilon^{-2})$ is almost attained when using QMC methods in both MLMC and smoothed MLMC, even in moderate high dimensions.
\end{abstract}



\section{Introduction}
We consider the problem of estimating
\begin{equation}
\theta = \mbp (g(\omega)>c)=\mbe[\mbi \{g(\omega)>c\}]=\mbe[\mbi \{\mbe [X|\omega]>c\}]\label{eq:theta}
\end{equation}
via simulation for a given constant $c$,
where $g(\omega):=\mbe [X|\omega]$. The inner expectation of the one-dimensional random variable, $X$, is conditional on the value of the outer multidimensional random variable $\omega$. This nested expectation appears, for instance,
when estimating the probability of a large loss from a financial portfolio. If the portfolio consists of some complex financial derivatives (such as path-dependent options and exotic options) or the underlying model is complicated, the future loss $g(\omega)$ of the portfolio on a fixed period is only known as a conditional expectation with respect to the risk factor $\omega$, which does not have an analytical form of $\omega$. A typical approach to estimate such an expectation of a function of a conditional
expectation is to use nested estimation. Specifically, nested simulation refers
to a two-level simulation procedure. In the outer level,
one generates a number of scenarios of $\omega$. Then, in
the inner level, a number of samples of
$X$ are generated for each simulated $\omega$ to estimate the conditional expectation $\mbe [X|\omega]$; see \cite{broa:2011} and \cite{GJ10}.

Nested simulation has been widely studied in the literature due to its broad applicability, particularly in portfolio risk measurement. Gordy and Juneja \cite{GJ10} considered and analyzed uniform
nested simulation estimators which employ a constant number of inner samples across all scenarios in the outer level.
They showed that to achieve a root mean squared error (RMSE) of $\epsilon$, allocating proper computational effort in each level results in a total  cost of $O(\epsilon^{-3})$. Nested simulation can be made
more efficient by allocating computational
effort nonuniformly across scenarios. Particularly, Broadie et al. \cite{broa:2011} proposed an adaptive procedure to allocate computational effort to inner simulations. The resulting nonuniform nested simulation estimator enjoys a reduced complexity of $O(\epsilon^{-5/2})$ under certain conditions. Recently, Giles and Haji-Ali \cite{giles:2018} proposed to use multilevel Monte Carlo (MLMC) method in nested simulation. They showed that in the original MLMC, the complexity is $O(\epsilon^{-5/2})$. By incorporating the adaptive allocations procedure of \cite{broa:2011}, Giles and Haji-Ali \cite{giles:2018} showed that the complexity of MLMC can be reduced to $O(\epsilon^{-2}(\log \epsilon)^2)$ under certain conditions. For some applications of nested MLMC, we refer to \cite{giles:2018b,GilesGoda19,goda:2020,Goda18} and references therein.

MLMC is a sophisticated variance reduction technique introduced by Heinrich \cite{Hein1998} for parametric integration and by Giles \cite{Giles2008} for the estimation of the expectations arising from stochastic differential equations. Nowadays MLMC method has been extended extensively. For a thorough review of MLMC methods, we refer to \cite{Giles2015}. On the other hand, quasi-Monte Carlo (QMC) and randomized QMC (RQMC) methods are alternatives to improve the efficiency of traditional Monte Carlo; see \cite{Dick2013,lecu:2002,Nied1992} for details. QMC methods have achieved great success in finance applications, such as option pricing and hedging \cite{lecu:2009}. It is natural to incorporate (R)QMC in the MLMC framework. Giles and Waterhouse \cite{Giles2009} first attempted to apply QMC method into multilevel path simulation in financial problems. In recent years, multilevel QMC methods have received increasing attention
amongst researchers due to its broad applicability, particularly in problems of partial differential equations with random coefficients, see e.g., \cite{Dick2016,Kuo2017,Kuo2015} and in uncertainty quantification, see e.g., \cite{dick:2017,sch:2017}.

In this paper, we focus on the combination of MLMC and QMC in nested simulation. Specifically, in the outer simulation, we use Monte Carlo to generate a number of scenarios of $\omega$. But, in the inner simulation, we use QMC to estimate the portfolio loss in each scenario.
Our work is closely related to Goda et al. \cite{Goda18}, who used nested multilevel RQMC method to deal with the expected value of partial perfect information (EVPPI) problem. A central problem of EVPPI is to estimate an expectation of the form $\mbe[f(g(\omega))]$, where $g(\omega)=\mbe[X|\omega]$ and $f(\cdot)$ is a continuous function. As Goda et al. \cite{Goda18} pointed out, the multilevel RQMC estimator can achieve the optimal complexity of $O(\epsilon^{-2})$ under some mild conditions. However, for the problem \cref{eq:theta} considered in this paper, the performance function $f(\cdot)$ becomes an indicator function $\mbi(\cdot>c)$, making it much
harder than EVPPI  for MLMC algorithms \cite{giles:2018b}. Moreover, the antithetic
MLMC estimator for EVPPI used in \cite{Goda18} does not help to reduce
the variance convergence rate in our setting because the discontinuity in the indicator
function violates the differentiability requirements of the antithetic estimator \cite{giles:2018}. As a result, the efficiency of using the antithetic form in multilevel RQMC is subtle for estimating \cref{eq:theta}. The gain of using RQMC in nested MLMC is also unclear. Can the multilevel RQMC estimator achieve the optimal complexity of $O(\epsilon^{-2})$ or the sub-optimal complexity of $O(\epsilon^{-2}(\log \epsilon)^2)$? It is well-known that the performance of (R)QMC integration depends on the smoothness of the integrand and the dimension of the problem \cite{He2018,He2019,He2015,xie:2019}. So it is natural to ask how these factors affect the efficiency of the multilevel RQMC algorithm.

The main contribution of this work is to reduce the complexity of the MLMC method for financial risk management by using QMC methods in the multilevel scheme. We also discuss some considerations of using antithetic MLMC with RQMC.  In \cref{sec:review}, we review QMC methods and nested simulation. The complexity of the uniform nested simulation combined with QMC is analyzed. After introducing the basic MLMC method, we develop a multilevel QMC procedure for the problem \cref{eq:theta}. The effects of using the QMC method and the indicator function on the nested simulation are discussed. In \cref{sec:smooth}, we develop a new smoothed coupling method, which aims to overcome the catastrophic coupling caused by small differences between the ``coarse" and ``fine" estimates
for the conditional expectation. Some numerical experiments are performed in \cref{sec:numerical}. Finally, we conclude this
paper with some remarks in \cref{sec:conclusion}.

\section{Nested simulation and Multilevel Monte Carlo}
\label{sec:review}

\subsection{Nested simulation and quasi-Monte Carlo}
Our goal is to estimate an expectation of a function of a conditional
expectation via nested simulation, with the problem \cref{eq:theta} of interest.
In the uniform nested simulation, one takes
\begin{displaymath}
\hat \theta_{n,m} = \frac 1 n \sum_{i=1}^n \mbi \{\hat g_m(\omega_i)>c\},
\end{displaymath}
and
\begin{equation}
\hat g_m(y) = \frac 1 m\sum_{j=1}^m X_j(y),\label{eq:innersim}
\end{equation}
where $\omega_i$ are independent and identically distributed (iid) replications of $\omega$, and $X_j(y)$ are iid replications of $X$ given $\omega=y$. The quadrature rule $\hat g_m(y)$ is used to estimate the conditional expectation $g(y)=\mbe[X|\omega=y]$ in the inner simulation.

In this paper, we incorporate QMC methods within the inner simulation. Let $\Omega$ be the possible scenarios of the random variable $\omega$. To fit MC or QMC framework, we assume that given $\omega=y$, $X$ can be generated via the mapping
\begin{equation}\label{eq:psi}
X(y)= \psi(\bm u;y),\ y\in\Omega,
\end{equation}
for some function $\psi$, where $\bm u\sim \Unif{[0,1)^d}$. As a result, the inner estimator \cref{eq:innersim} is replaced by
\begin{equation}
\hat g_m(y) = \frac 1 m\sum_{j=1}^m \psi(\bm u_j;y),\label{eq:innerqmcsim}
\end{equation}
and $g(y) = \int_{[0,1)^d} \psi(\bm u;y) \mrd \bm u$.
If $\bm u_1,\dots,\bm u_m$ are iid uniform points over $[0,1)^d$, the approximation \cref{eq:innerqmcsim} refers to the MC method, which attains a probabilistic convergence rate of $O(m^{-1/2})$. If $\bm u_1,\dots,\bm u_m$ are QMC points (known as low discrepancy points), which are deterministic points chosen from $[0,1)^d$ and are more uniformly distributed than random points, the approximation \cref{eq:innerqmcsim} refers to the QMC method.

The error of the QMC quadrature can be bounded by the Koksma-Hlawka inequality (see \cite{Nied1992})
\begin{equation}
|g(y)-\hat g_m(y)|\le V_{\mathrm{HK}}\big(\psi(\cdot;y)\big)D^*(\bm u_1,\dots,\bm u_m),\label{eq:khineq}
\end{equation}
where $V_{\mathrm{HK}}\big(\psi(\cdot;y)\big)$ is the variation of the integrand $\psi(\cdot;y)$ for given $y$ in the sense of Hardy and Krause which measures the smoothness of $\psi(\cdot;y)$, and $D^*(\bm u_1,\dots,\bm u_m)$ is the star discrepancy which measures the uniformity of the points set $\{\bm u_1,\dots,\bm u_m\}$. The Koksma-Hlawka inequality \cref{eq:khineq} implies that for functions of finite variation, the convergence rate of QMC approximation is determined by the factor $D^*(\bm u_1,\dots,\bm u_m)$, which is of order $O(m^{-1}(\log m)^d)$ for low discrepancy points.

There are various constructions of QMC point sets in the literature \cite{dick:2017}, such as digital nets and lattice rule point sets. In this paper, we use $(t,d)$-sequences in base $b\ge 2$.
In practice, one uses RQMC points for ease of evaluating quadrature errors. RQMC retains the essential equi-distribution structure, while allowing a statistical error estimation based on independent replications. Moreover, RQMC is able to improve the rate of convergence for smooth integrands, such as the Owen's scrambling technique \cite{Owen1995}. For a survey on RQMC, we refer to \cite{lecu:2002}. It should be noted that, in both MC and RQMC settings, $\mbe[\hat g_m(y)]=g(y)$.

\begin{assum}\label{assm1}
	There exist a constant $\eta\ge 1$ and a function $\sigma^2(\cdot )\ge 0$ such that for any $y\in \Omega$,
	\begin{displaymath}
	\var{\hat g_m(y)} \le \frac{\sigma^2(y)}{m^\eta}.
	\end{displaymath}
\end{assum}

In the MC setting, we can take $\eta=1$ and $\sigma^2(y)=\var{X|\omega=y}$. In the RQMC setting, we can expect a larger value of $\eta$.  Indeed, a digit scrambling of \cite{Owen1995} applied to $(t, d)$-sequences leads to a variance of $O(m^{-2}(\log m)^{2d})$ for integrands of finite variation in the sense of Hardy and Krause, and a variance of $O(m^{-3}(\log m)^{d-1})$ for smooth enough integrands \cite{Owen1998}. Furthermore, for any square integrable integrand, the scrambled net variance $\var{\hat g_m(y)}$ has a conservative upper bound $M\var{X|\omega=y}/m$, where the constant $M$ depends on $t,\ d$ and $b$ \cite{Owen1997}. As a result, it is reasonable to assume that $\eta\ge 1$ for RQMC. If the function $\psi$ in \cref{eq:innerqmcsim} is sufficiently smooth, then $\eta\approx 2$ or even larger.

The next theorem is a generalization of Proposition 1 in \cite{GJ10}.
\begin{theorem}\label{thm:bias}
	Suppose that \cref{assm1} is satisfied, and let $f(\cdot)$ be the probability density function of $g(\omega)$. Assume the following:
\begin{itemize}
	\item The joint density $p_m(x,y)$ of $g(\omega)$ and $m^{\eta/2}[\hat{g}_m(\omega)-g(\omega)]$ and partial derivatives $(\partial /\partial x)p_m(x,y)$ and $(\partial^2 /\partial x^2)p_m(x,y)$ exist for each $m$ and $(x,y)$.
	\item For each $m\ge 1$, there exist functions $f_{i,m}(\cdot)$ such that
\begin{displaymath}
\bigg|\frac{\partial^i }{\partial x^i}p_m(x,y)\bigg|\le f_{i,m}(y), i=0,1,2.
\end{displaymath}
	In addition,
\begin{displaymath}	
\sup_{m\ge 1} \int \abs{y}^r f_{i,m}(y) \mrd y<\infty,\ for\ i=0,1,2\ and\ 0\le r\le 4.
\end{displaymath}
\end{itemize}
Then the bias of the nested estimator $\hat \theta_{n,m}$ asymptotically satisfies
\begin{equation}\label{eq:bias}
|\mbe[\hat \theta_{n,m}-\theta] |\le  \frac{|\Theta'(c)|}{m^{\eta}} + O(m^{-3\eta/2}),
\end{equation}
where
\begin{displaymath}
\Theta(c)=\frac 1 2 f(c)\mbe[\sigma^2(\omega)|g(\omega)=c],
\end{displaymath}
and $\sigma^2(y)$ is given in \cref{assm1}.
\end{theorem}
\begin{proof}
Let
\begin{displaymath}
\theta_m:=\mbe(\hat\theta_{n,m})=\mbp(\hat g_m(\omega)>c).
\end{displaymath}
Note that
\begin{displaymath}
\theta_m=\int_\mathbb{R}\int_{c-ym^{-\eta/2}}^\infty p_m(x,y)dxdy,
\end{displaymath}
so that
\begin{displaymath}
\theta_m-\theta=\int_\mathbb{R}\int_{c-ym^{-\eta/2}}^c p_m(x,y)dxdy.
\end{displaymath}
Consider the Taylor series expansion of the density function $p_m(x,y)$ at $x=c$,
\begin{displaymath}
p_m(x,y)=p_m(c,y)+(x-c)\frac{\partial}{\partial x}p_m(c,y)+\frac{(x-c)^2}{2}\frac{\partial^2}{\partial x^2}p_m(\tilde c,y),
\end{displaymath}
where $\tilde c$ is a number between $c$ and $x$. From this expansion and the assumptions in the theorem, it follows that
\begin{equation}\label{eq:dm}
\theta_m-\theta=\int_\mathbb{R}\frac{y}{m^{\eta/2}}p_m(c,y)dy
-\int_\mathbb{R}\frac{y^2}{2m^\eta}\frac{\partial}{\partial x}p_m(c,y)dy+O(m^{-3\eta/2}).
\end{equation}
For the first term on the right hand side of \cref{eq:dm}, we have
\begin{align*}
\int_\mathbb{R}\frac{y}{m^{\eta/2}}p_m(c,y)dy&=\frac{f(c)}{m^{\eta/2}}\int_\mathbb{R}y\cdot\frac{p_m(c,y)}{f(c)}dy\\
&=\frac{f(c)}{m^{\eta/2}}\mbe[m^{\eta/2}(\hat{g}_m(\omega)-g(\omega))|g(\omega)=c].
\end{align*}
This term is zero because
\begin{displaymath}
\mbe[m^{\eta/2}(\hat{g}_m(\omega)-g(\omega))|g(\omega)=c]
=\mbe[\mbe[m^{\eta/2}(\hat{g}_m(\omega)-g(\omega))|g(\omega)=c,\omega]]=0.
\end{displaymath}
For the second term on the right hand side of \cref{eq:dm}, we have
\begin{align*}
\int_\mathbb{R}\frac{y^2}{2m^\eta}\frac{\partial}{\partial x}p_m(c,y)dy&=\frac{1}{2m^\eta}\frac{d}{dc}\int_{\mathbb{R}}y^2p_m(c,y)dy\\
&=\frac{1}{2m^\eta}\frac{d}{dc}\bigg[f(c)\int_{\mathbb{R}}y^2\cdot\frac{p_m(c,y)}{f(c)}dy\bigg]\\
&=\frac{1}{2m^\eta}\frac{d}{dc}\bigg[f(c)\mbe[m^{\eta}(\hat{g}_m(\omega)-g(\omega))^2|g(\omega)=c]\bigg]\\
&=\frac{1}{2m^\eta}\frac{d}{dc}\bigg[f(c)\mbe[\mbe[m^{\eta}(\hat{g}_m(\omega)-g(\omega))^2|\omega]|g(\omega)=c]\bigg]\\
&=\frac{1}{2m^\eta}\frac{d}{dc}\bigg[f(c)\mbe[\sigma^2(\omega)|g(\omega)=c]\bigg]\\
&=\frac{\Theta'(c)}{m^{\eta}}
\end{align*}
By \eqref{eq:dm}, the proof is completed.
\end{proof}

We now use \cref{thm:bias}  to analyze the mean squared error (MSE) of the nested estimator. By \cref{eq:bias}, it is easy to see that
\begin{displaymath}
\begin{aligned}
\var{\hat \theta_{n,m}}&=\frac{1}{n}\var{\mbi(\hat g_m(\omega_1)>c)}=\frac{\theta_m(1-\theta_m)}{n}\\
&=\frac{\theta(1-\theta)}{n}+\frac{(\theta_m-\theta)(1-\theta_m)}{n}+\frac{\theta(\theta-\theta_m)}{n}\\
&\le \frac{\theta(1-\theta)}{n} + O(n^{-1}m^{-\eta}).
\end{aligned}
\end{displaymath}
So we have
\begin{align*}
\mbe[(\hat \theta_{n,m}-\theta)^2] &= \var{\hat \theta_{n,m}}+(\mbe[\hat \theta_{n,m}-\theta])^2 \\
&\le \frac{\theta(1-\theta)}{n} + \frac{\Theta'(c)^2}{m^{2\eta}} + O(m^{-5\eta/2}) + O(n^{-1}m^{-\eta}).
\end{align*}

To achieve an RMSE of $\epsilon$, this suggests the optimal allocations $n=O(\epsilon^{-2})$ and $m=O(\epsilon^{-1/\eta})$. The total computational cost is then $O(\epsilon^{-2-1/\eta})$. In the MC setting where $\eta=1$, the complexity is $O(\epsilon^{-3})$, see \cite{GJ10}. In RQMC setting, it is possible to achieve $\eta\approx 2$, so that the complexity is around $O(\epsilon^{-5/2})$. Broadie et al. \cite{broa:2011} used adaptive allocations in the inner simulation to reduce the nested MC down to $O(\epsilon^{-5/2})$. Giles and Haji-Ali \cite{giles:2018} showed that in the original MLMC, the  complexity is also $O(\epsilon^{-5/2})$. By incorporating the adaptive allocations idea of \cite{broa:2011}, Giles and Haji-Ali \cite{giles:2018} showed that the complexity of MLMC can be reduced to $O(\epsilon^{-2}(\log \epsilon)^2)$.
We next focus on using the framework of multilevel method to reduce the complexity of RQMC-based nested simulation.

\subsection{MLMC estimators}\label{mlmcestimators}
We review briefly the basic idea of MLMC. Let $P=\theta = \mbi\{g(\omega)>c\}$, which requires an infinite cost to evaluate. Instead of dealing with $P$ directly, we consider a sequence of random variables $P_0,P_1,\dots$ with increasing approximation accuracy to $P$ but with increasing cost per sample. The $\ell$th level approximation of $P$ is defined as $P_\ell=\mbi\{\hat{g}_{m_\ell}(\omega)>c\}$, where $m_\ell=2^{\ell+\ell_0}$ in our setting and $\ell_0\ge 0$. Due to the linearity of expectation, we have the following telescoping representation
\begin{equation}\label{eq:telescope}
\mbe[P_L] = E[P_0] + \sum_{\ell=1}^L \mbe[P_{\ell}-P_{\ell-1}].
\end{equation}
Let $Y_\ell=P_{\ell}-P_{\ell-1}$ for $\ell \ge 1$ and $Y_0=P_0$, we further have
\begin{equation}\label{eq:sumeq}
\mbe[P_L] = \sum_{\ell=0}^L \mbe[Y_\ell].
\end{equation}
The MLMC method uses the above equality and estimates each term on the right hand side of \cref{eq:sumeq} independently. The resulting MLMC estimator is given by
$$\mlmc = \sum_{\ell=0}^L Z_\ell$$
with
$$Z_\ell = \frac{1}{N_\ell}\sum_{i=1}^{N_\ell}Y_\ell^{(i)},$$
where $Y_\ell^{(1)},\dots,Y_\ell^{(N_\ell)}$ are iid replications of $Y_\ell$ for $\ell=0,\dots, L$. Denote the variance and the computational cost of $Y_\ell$ by $V_\ell$ and $C_\ell$, respectively. The MSE of $\mlmc$ is then given by

\begin{displaymath}
\mbe[(\mlmc-\theta)^2] = \sum_{\ell=0}^L \frac{V_\ell}{N_\ell} + (\mbe[P_L-\theta])^2.
\end{displaymath}
The total cost of $\mlmc$ is $C=\sum_{\ell=0}^LN_\ell C_\ell$.

The pioneering work by Giles \cite{Giles2008} established the following theorem for MLMC.
\begin{theorem}\label{thm:mlmc}
If there are constants $\alpha,\beta,\gamma,c_1,c_2,c_3$ such that $\alpha\ge \min(\beta,\gamma)/2$ and
\begin{itemize}
	\item $|\mbe[P_\ell -\theta]|\le c_1 m_\ell^{-\alpha}$,
	\item $V_\ell \le c_2 m_\ell^{-\beta}$,
	\item $C_\ell \le c_3m_\ell^{\gamma}$,
\end{itemize}
then there exists a constant $c_4>0$ such that for any $\epsilon<1/e$, MLMC estimator $\mlmc$ has an MSE bound $\mbe[(\mlmc-\theta)^2]\le \epsilon^2$ with a total computational cost $C$ bounded by
\begin{displaymath}
C\le \begin{cases}
c_4\epsilon^{-2}, &\beta>\gamma,\\
c_4\epsilon^{-2}(\log \epsilon)^2, &\beta=\gamma,\\
c_4\epsilon^{-2-(\gamma-\beta)/\alpha}, &\beta<\gamma.
\end{cases}
\end{displaymath}
\end{theorem}

The constants $\alpha$ and $\beta$ in \cref{thm:mlmc} describe the rates of bias and variance decreasing, respectively, which are usually called \textit{weak convergence} and \textit{strong convergence}. The constant $\gamma$ controls the increasing of computing budget for each sample.

Now we develop our nested multilevel QMC estimator. Here we use the following form for coupling the two consecutive levels: $Y_0=P_0=\mbi\{\hat{g}_{m_0}(\omega)>c\}$, and for $\ell\ge 1$,

\begin{displaymath}
Y_\ell =P_\ell-P_{\ell-1} =\mbi\{\hat{g}_{m_\ell}(\omega)>c\}-\mbi\{\hat{g}_{m_{\ell-1}}(\omega)>c\},
\end{displaymath}
where
\begin{displaymath}
\hat g_{m_\ell}(y) = \frac 1 {m_\ell}\sum_{j=1}^{m_\ell} \psi(\bm u_j;y),
\end{displaymath}
and $\psi$ is given by \cref{eq:psi}.

It is easy to see that $\mbe[Y_\ell]=\mbe[P_{\ell}-P_{\ell-1}]$ for $\ell \ge 1$. Also, $C_\ell =  O(m_\ell)$, i.e., $\gamma=1$. \cref{thm:bias} suggests that $\alpha= \eta\ge 1/2$ in \cref{thm:mlmc}.

We next study the variance of $Y_\ell$. Note that for $\ell\ge 1$, we have
\begin{equation}
\var{Y_\ell}\le 2\var{P_{\ell}-P} +2\var{P_{\ell-1}-P},\label{eq:vars}
\end{equation}
since $\var{A+B}\le 2[\var{A}+\var{B}]$ for any random variables $A,B$ with finite variances. It thus suffices to study the decay rate of $\var{P_{\ell}-P}$, which can be described by value of the constant $\beta$.

\begin{assum}\label{assum:var}
	Assume that the density of the random variable $|g(\omega)-c|/\sigma(\omega)$, where $\sigma(\omega)$ is the square root of $\sigma^2(\omega)$ in \cref{assm1}, denoted by $\rho(\cdot)$, exists. Moreover, we assume that there exist constants $\rho_0>0$ and $x_0>0$ such that $\rho(x)\le \rho_0$ for all $x\in[0,x_0]$.
\end{assum}

The next theorem is a generalization of Proposition 2.2 in \cite{giles:2018}.

\begin{theorem}\label{thm:beta}
Suppose that \cref{assm1} and \cref{assum:var} are satisfied. Then
	\begin{displaymath}
	\var{P_{\ell}-P}\le \mbe[(P_{\ell}-P)^2] = O(m_\ell^{-\eta/2}).
    \end{displaymath}
\end{theorem}
\begin{proof}
Let us start from
\begin{align*}
\mbe\bigg[(\mbi\{\hat{g}_{m_\ell}(\omega)>c\}-\mbi\{g(\omega)>c\})^2\big|\omega\bigg]
&=\mathbb{P}\bigg[|\mbi\{\hat{g}_{m_\ell}(\omega)>c\}-\mbi\{g(\omega)>c\}|=1\big|\omega\bigg]\\
&\le\mathbb{P}\bigg[|\hat{g}_{m_\ell}(\omega)-g(\omega)|\ge|g(\omega)-c|\big|\omega\bigg].
\end{align*}
By Chebyshev's inequality, we have
\begin{align*}
\mathbb{P}\bigg[|\hat{g}_{m_\ell}(\omega)-g(\omega)|\ge|g(\omega)-c|\big|\omega\bigg]
&\le\min\bigg(1,|g(\omega)-c|^{-2}\var{\hat{g}_{m_\ell}(\omega)\big|\omega}\bigg)\\
&\le\min\bigg(1,|g(\omega)-c|^{-2}\frac{\sigma^2(\omega)}{m_\ell^\eta}\bigg).
\end{align*}
Taking expectation over $\omega$ yields
\begin{align*}
\mbe\bigg[(\mbi\{\hat{g}_{m_\ell}(\omega)>c\}-\mbi\{g(\omega)>c\})^2\bigg] &
\le\int_0^\infty\min(1,x^{-2}m_\ell^{-\eta})\rho(x)dx\\
&\le\rho_0\int_0^\infty \min(1,x^{-2}{m_\ell^{-\eta}})dx+\min(1,x_0^{-2}m_\ell^{-\eta})\\
&\le 2\rho_0m^{-\eta/2}_\ell+x^{-2}_0m^{-\eta}_\ell.
\end{align*}
Hence, the variance is $O(m^{-\eta/2}_\ell)$.
\end{proof}

Based on \cref{eq:vars} and \cref{thm:beta}, we have $\var{Y_\ell}=O(m_\ell^{-\beta})$ with $\beta = \eta/2$. Since $\gamma=1$, by \cref{thm:mlmc}, the total cost of MLQMC thus becomes
\begin{equation}\label{eq:complexity}
C= \begin{cases}
O(\epsilon^{-2}), &\eta>2,\\
O(\epsilon^{-2}(\log \epsilon)^2), &\eta=2,\\
O(\epsilon^{-3/2-1/\eta}), &\eta\in[1,2).
\end{cases}
\end{equation}

It is a common strategy to use an antithetic form for coupling the consecutive levels, i.e.,
\begin{displaymath}
Y_\ell = \mbi\{\hat{g}_{m_\ell}(\omega)>c\}-\frac{1}{2}\left(\mbi\{\hat{g}_{m_{\ell-1}}^{(1)}(\omega)>c\}+\mbi\{\hat{g}_{m_{\ell-1}}^{(2)}(\omega)>c\}\right),
\end{displaymath}
where
\begin{displaymath}
\hat{g}_{m_{\ell-1}}^{(i)}(\omega)=\frac 1 {m_{\ell-1}}\sum_{j=1+(i-1)m_{\ell-1}}^{im_{\ell-1}} \psi(\bm u_j;\omega),\ i=1,2.
\end{displaymath}
Antithetic sampling can sometimes reduce variance apparently, see \cite{GilesGoda19} and \cite{Goda18}. But it should be cautious to adapt antithetic sampling when the underlying function is discontinuous. It is required that
$$\mbp[\hat{g}_{m_{\ell-1}}^{(1)}(\omega)>c|\omega]= \mbp[\hat{g}_{m_{\ell-1}}^{(2)}(\omega)>c|\omega]$$
in order to ensure the telescoping representation \cref{eq:telescope} under RQMC scheme.  This can be achieved if the first half RQMC points $\bm u_1,\dots,\bm u_{m_{\ell-1}}$ have the same joint distribution as that of the second half RQMC points $\bm u_{m_{\ell-1}+1},\dots,\bm u_{m_{\ell}}$. As pointed out by Goda et al. \cite{Goda18}, the well-known explicitly constructed $(t, d)$-sequences in base $b=2$ satisfy this condition if using Owen's scrambling method \cite{Owen1995}. However, antithetic sampling does not change the variance convergence rate in this problem because the discontinuity in the indicator function violates the differentiability requirements of the antithetic estimator.

\section{Smoothed MLMC method}\label{sec:smooth}
Recall that
\begin{displaymath}
Y_\ell =P_\ell-P_{\ell-1} =\mbi\{\hat{g}_{m_\ell}(\omega)>c\}-\mbi\{\hat{g}_{m_{\ell-1}}(\omega)>c\}.
\end{displaymath}
The ``fine" estimator $\hat g_{m_\ell}(\omega)$ and the ``coarse" estimator $\hat g_{m_{\ell-1}}(\omega)$ conditional the same scenario $\omega$ have small differences in most situations, especially for large $\ell$, making $Y_\ell$ different from zero only in a tiny proportion of the scenarios. Moraes et al. \cite{Alvaro2016} called this phenomenon as \textit{catastrophic coupling}.

This may lead to a higher kurtosis, which is defined by
\begin{displaymath}
\kappa=\frac{\mbe[(Y_\ell-\mbe[Y_\ell])^4]}{(\var {Y_\ell})^2}.
\end{displaymath}
This is a challenge for both MLMC and MLQMC. As shown in Hammouda et al. \cite{Hamm2020}, the sample variance of $Y_\ell^{(1)},\dots,Y_\ell^{(N_\ell)}$ has a standard error of
\begin{displaymath}
\sigma_{\ell}:=\frac{\var{Y_\ell}}{\sqrt {N_\ell}}\sqrt{\kappa-1+\frac{2}{N_\ell-1}}.
\end{displaymath}
And the high kurtosis makes it challenging to estimate $V_\ell$ accurately in deeper levels, since there are less samples in deeper levels under MLMC scheme \cite{Hamm2020}.

In fact, the high-kurtosis phenomenon is inevitable due to the indicator function: $P_\ell$ takes values in $\{0,1\}$, then $Y_\ell=P_\ell-P_{\ell-1}$ ($\ell\ge1$) takes values in $\{-1,0,1\}$. As a result, the iid replications of $Y_\ell$, $Y^{(1)}_\ell,\dots,Y^{(N_\ell)}_\ell$ take values in $\{-1,0,1\}$ as well. The sample mean of  $Y^{(1)}_\ell,\dots,Y^{(N_\ell)}_\ell$, denoted by $\bar Y_\ell$, is usually negligible because of small differences between $P_\ell$ and $P_{\ell-1}$. We find that the sample kurtosis
\begin{align}
\hat\kappa&=\frac{\frac{1}{N_\ell}\sum_{i=1}^{N_\ell}(Y_\ell^{(i)}-\bar Y_\ell)^4}{(\frac{1}{N_\ell}\sum_{i=1}^{N_\ell}(Y_\ell^{(i)}-\bar Y_\ell)^2)^2}
\label{eq:kv}\\&\approx\frac{\frac 1 {N_\ell} \sum_{i=1}^{N_\ell} {Y^{(i)}_\ell}^4}{(\frac 1 {N_\ell}\sum_{i=1}^{N_\ell}{Y^{(i)}_\ell}^2)^2}=\frac{1}{\frac 1 {N_\ell}\sum_{i=1}^{N_\ell}{Y^{(i)}_\ell}^2}\approx\frac 1 {\mathcal{S}^2_\ell},\label{eq:orik}
\end{align}
where $\mathcal{S}^2_\ell=\frac 1 {N_\ell}\sum_{i=1}^{N_\ell}(Y^{(i)}_\ell-\bar Y_\ell)^2$ is the sample variance.
So the kurtosis increases with the same rate as the variance decreases.

It should be noted that this phenomenon is due to the structure of the indicator function. Moreover, the indicator function eliminates further the difference between $\hat g_{m_\ell}(\omega)$ and $\hat g_{m_{\ell-1}}(\omega)$, then this aggravates the catastrophic coupling. This phenomenon enlightens us to replace the indicator function with a smoother function. To this end, we introduce the logistic sigmoid function
\begin{displaymath}
S(x)=\frac{1}{1+\exp(-x)},\ x\in\mathbb{R}.
\end{displaymath}
Logistic sigmoid function is widely applied in traditional classification task, acting as activation function in deep learning, see e.g. \cite{Hastie2016} and \cite{LeCun2015}. Based on this function, we construct a family of sigmoid-like functions
\begin{displaymath}
S(x;k)=\frac{1}{1+\exp(-kx)},\ x\in\mathbb{R},
\end{displaymath}
 where $k>0$.
It is obvious that $S(x;k)$ has the following properties:
\begin{itemize}
	\item $S(x;k)$ takes value in $(0,1)$ for any $k$;
	\item $S(x;k)$ is centrosymmetric about $S(0;k)=\frac{1}{2}$ for any $k$;
	\item $S(x;k)$ converges to $\mbi\{x>0\}$ as $k\rightarrow \infty$ for any fixed $x$.
\end{itemize}

\cref{fig:sigmoid} presents several examples of $S(x;k)$. It can be seen that the sigmoid-like functions approximate the indicator function. When the value of $k$ gets larger, the approximation of $S(x;k)$ to $\mbi(x>0)$ gets better.
\begin{figure}[htbp]
  \centering
  \includegraphics[width=8cm]{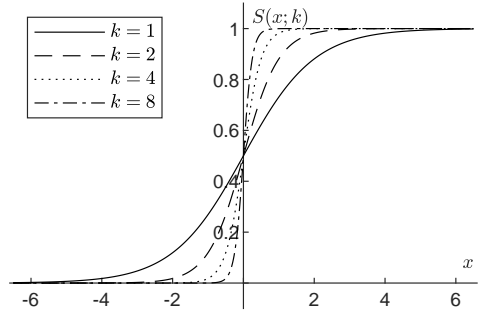}
  \caption{$S(x;k)$ with different $k$s.}
  \label{fig:sigmoid}
\end{figure}

Now we are ready to develop a smoothed MLMC method.  Define the $\ell$th level approximation of $P=\mbi\{g(\omega)>c\}$ as
\begin{displaymath}
\tilde P_\ell=S^{(k_0,r)}_\ell(\hat{g}_{m_\ell}(\omega)-c),
\end{displaymath}
 instead of $P_\ell=\mbi\{\hat{g}_{m_\ell}(\omega)>c\}$ used in \cref{mlmcestimators}, where $S^{(k_0,r)}_\ell(x)=S(x;k_0r^\ell)$, $k_0>0$ is a constant and $r>1$ controls the rate of $S^{(k_0,r)}_\ell(x)$ converging to $\mbi\{x>0\}$ as $\ell\to \infty$.

Similar to \cref{eq:telescope}, we have the following telescoping representation
\begin{displaymath}
\mbe[\tilde P_L] = E[\tilde P_0] + \sum_{\ell=1}^L \mbe[\tilde P_{\ell}-\tilde P_{\ell-1}].
\end{displaymath}
Let $\tilde Y_\ell=\tilde P_{\ell}-\tilde P_{\ell-1}$ for $\ell \ge 1$ and $\tilde Y_0=\tilde P_0$. It is obvious that
\begin{displaymath}
\tilde P_L = \sum_{\ell=0}^L \tilde Y_\ell,
\end{displaymath}
and
\begin{displaymath}
\mbe[\tilde P_L] = \sum_{\ell=0}^L \mbe[\tilde Y_\ell].
\end{displaymath}
This yields a new MLMC estimator
\begin{displaymath}
\seqmlmc = \sum_{\ell=0}^L \tilde Z_\ell
\end{displaymath}
with
\begin{displaymath}
\tilde Z_\ell = \frac{1}{N_\ell}\sum_{j=1}^{N_\ell}\tilde Y_\ell^{(i)},
\end{displaymath}
where $\tilde Y_\ell^{(1)},\dots,\tilde Y_\ell^{(N_\ell)}$ are iid replications of $\tilde Y_\ell$ for $\ell=0,\dots, L$.
Notice that $\tilde P_L$ converges to $P$ as $L$ goes to infinity, just like $P_L$ does.
Denoting the variance and cost of each sample of $\tilde Y_\ell$ by $\tilde V_\ell$ and $\tilde C_\ell$, the MSE of $\seqmlmc$ is given by
\begin{displaymath}
\mbe[(\seqmlmc-\theta)^2] = \sum_{\ell=0}^L \frac{\tilde V_\ell}{N_\ell} + (\mbe[\tilde P_L-\theta])^2.
\end{displaymath}
The total cost of $\seqmlmc$ is $\tilde C=\sum_{\ell=0}^LN_\ell \tilde C_\ell$. It is critical to verify the strong convergence. To this end, we make the following assumption.
\begin{assum}\label{assum:density}
Let $\Pi_\ell(x)$ denote the the cumulative distribution function of the random variable $\hat g_{m_\ell}(\omega)-c$. Assume that the densities of these variables, denoted by $\pi_\ell(x)$ for $\ell=0,1,\dots$, exist in a common neighborhood of the origin, say $(-\delta_0,\delta_0)$. Assume also that there exists a maximum $\pi^*_\ell=\sup_{x\in(-\delta_0,\delta_0)}\pi_\ell(x)<+\infty$ for all $\ell$ and there is a constant $C_\pi$ such that $\sup_{\ell}\pi_{\ell}^*<C_\pi$.
\end{assum}
The first part of \cref{assum:density} is similar to the assumptions in \cref{thm:bias}, but here we only require the densities to exist in a neighborhood of the origin. For the later part of \cref{assum:density}, heuristically, we have $\lim_{\ell\rightarrow\infty}\pi_\ell(x)=\pi(x)$, where $\pi(x)$ is density of $g(\omega)$. Then we can expect that $\pi_\ell^*$ can be controlled by a constant $C_\pi$.
\begin{theorem}\label{thm:seqconvergence}
Suppose that the conditions of \cref{assm1}, \cref{assum:var} and \cref{assum:density} are satisfied. Then
\begin{displaymath}
	\var{\tilde P_{\ell}-P}= O(\max(r^{-\ell},m_\ell^{-\eta/2})).
\end{displaymath}
\end{theorem}
\begin{proof}
First, we have
\begin{align*}
\mbe[(\tilde P_\ell-P_\ell)^2]&=\int_{-\infty}^{-\delta_0}[S_\ell^{(k_0,r)}(x)]^2d\Pi_\ell(x)
+\int_{-\delta_0}^0[S_\ell^{(k_0,r)}(x)]^2\pi_\ell(x)dx \\
&+\int_{0}^{\delta_0}[S_\ell^{(k_0,r)}(-x)]^2\pi_\ell(x)dx
+\int_{\delta_0}^{+\infty}[S_\ell^{(k_0,r)}(-x)]^2d\Pi_\ell(x)\\
&\le 2[S_\ell^{(k_0,r)}(-\delta_0)]^2+2\pi^*_\ell\int_{-\delta_0}^0[S_\ell^{(k_0,r)}(x)]^2dx\\
&=\frac{2}{[1+\exp(k_0r^\ell\delta_0)]^2}+\bigg[\ln\frac 2 {1+e^{-k_0r^\ell\delta_0}}+\frac 1 2 -\frac{1}{1+e^{-k_0r^\ell\delta_0}}\bigg]\cdot\frac{2C_\pi}{k_0r^\ell}\\
&=O(r^{-\ell}).
\end{align*}

By \cref{thm:beta}, we have  $\mbe[(P_\ell-P)^2]=O(m_\ell^{-\eta/2})$.
It then follows that
\begin{displaymath}
\mbe[(\tilde P_\ell-P)^2]\le2\mbe[(\tilde P_\ell-P_\ell)^2]+2\mbe[(P_\ell-P)^2]=O(\max(r^{-\ell},m_\ell^{-\eta/2})).
\end{displaymath}
\end{proof}

Similarly, we can verify the weak convergence of the smoothed MLMC method.

\begin{theorem}\label{thm:weakconvergence}
Suppose that the conditions of \cref{thm:bias} and \cref{assum:density} are satisfied. Then
    \begin{equation}\label{eq:weakresult}
	|\mbe[\tilde P_\ell-\theta]|=O(\max(r^{-\ell},m_\ell^{-\eta})).
    \end{equation}
\end{theorem}

\begin{proof}
Note that
\begin{equation}\label{eq:weakconvergence}
|\mbe[\tilde P_\ell-\theta]|\le\mbe[|\tilde P_\ell-P_\ell|]+|\mbe[P_\ell-\theta]|.
\end{equation}
By \cref{thm:bias}, $|\mbe[P_\ell-\theta]|=O(m^{-\eta}_\ell)$. Under \cref{assum:density}, we have
\begin{displaymath}
\begin{aligned}
\mbe[|\tilde P_\ell-P_\ell|]&=\int_{-\infty}^{-\delta_0}S_\ell^{(k_0,r)}(x)d\Pi_\ell(x)
+\int_{-\delta_0}^0S_\ell^{(k_0,r)}(x)\pi_\ell(x)dx \\
&+\int_{0}^{\delta_0}S_\ell^{(k_0,r)}(-x)\pi_\ell(x)dx
+\int_{\delta_0}^{+\infty}S_\ell^{(k_0,r)}(-x)d\Pi_\ell(x)\\
&\le 2S_\ell^{(k_0,r)}(-\delta_0)+2\pi^*_\ell\int_{-\delta_0}^0S_\ell^{(k_0,r)}(x)dx\\
&=\frac{2}{1+\exp(k_0r^\ell\delta_0)}+\ln\frac 2 {1+\exp(-k_0r^\ell\delta_0)}\cdot\frac{2C_\pi}{k_0r^\ell}\\
&=O(r^{-\ell}).
\end{aligned}
\end{displaymath}
The result \cref{eq:weakresult} follows from \cref{eq:weakconvergence}.
\end{proof}

For a given $\eta$, it is sufficient to take $r=2^{\eta/2}$, and then $O(r^{-\ell})\le O(m_\ell^{-\eta/2})$, since $m_\ell=O(2^\ell)$. As a result, the strong convergence of the smoothed method is the same as the original ML(Q)MC methods while the weak convergence may slow down to $O(m^{-\eta/2}_\ell)$ because of the smoothed method. When $\eta=2$, we can take $r=2$ such that there are constants $c_2,c_3$ such that $\tilde V_\ell\le c_2m_\ell^{-1}$ and $\tilde C_\ell\le c_3m_\ell^{-1}$, which means $\beta=\gamma=1$, then the cost is in the second regime in \cref{thm:mlmc}, where weak convergence makes little difference. So $r=2$ is sufficient for this situation to get the expected improvement.

The next theorem can be obtained immediately from \cref{thm:mlmc}.
\begin{theorem}
 Suppose that \cref{assm1} is satisfied with $\eta=2$ and $r=2$. Then for any $\epsilon<1/e$, the smoothed MLMC estimator $\seqmlmc$ has an MSE bound $\mbe[(\seqmlmc-\theta)^2]\le \epsilon^2$ with a total computational cost of $O(\epsilon^{-2}(\log \epsilon)^2)$.
\end{theorem}

\cref{thm:seqconvergence} and \cref{thm:weakconvergence} guarantee that the smoothed MLMC method converges to the true value under a certain rate. However, it should be noticed that $r=2$ is sufficient but unnecessary for the smoothed MLMC method.
In fact, when the dimension is relatively high, we would like to take $r$ a relatively small value. On the one hand, the efficiency of QMC method is affected by high dimensionality leading an $\eta$ value smaller than 2. As a result, a smaller $r$ can also match to $O(m_\ell^{-\eta/2})$ rate. On the other hand, if the value of $r$ is smaller, the coupling between the ``fine" estimator $\tilde P_\ell$ and the ``coarse" estimator $\tilde P_{\ell-1}$ is better and then this leads to a smaller variance. By this way, we can get a better $\beta$.

In addition, the logistic sigmoid function $S_\ell^{(k_0,r)}(x)$ with a smaller $r$ can maintain milder derivative in the neighborhood of the origin which relieves the catastrophic coupling. Antithetic sampling can also benefit from the differentiability and mild derivative of the logistic sigmoid functions with small $r$. If $r$ is taken too large, $S_\ell^{(k_0,r)}(x)$ converges to indicator function rapidly, making the smoothed MLMC method inefficient. So in our numerical studies, we take $r=2$ in one-dimensional problems, while take $r=\sqrt 2$ in multiple-dimensional problems. We take $k_0=8$ in $\tilde P_\ell=S^{(k_0,r)}_\ell(\hat{g}_{m_\ell}(\omega)-c)$, which guarantees the sigmoid functions not differing from the indicator function too much at the beginning.

Finally, we define \textit{kurtosis and variance factor} (\textit{KVF}) of level $\ell$ as the product of kurtosis and variance in level $\ell$,
$$\textrm{KVF}=\kappa V_\ell.$$
KVF can be estimated by
$$\widehat{\textrm{KVF}}=\hat\kappa \mathcal{S}^2_\ell,$$
where  $\hat\kappa$ is  defined by \cref{eq:kv}, and $\mathcal{S}^2_\ell=\frac 1 {N_\ell}\sum_{i=1}^{N_\ell}(Y^{(i)}_\ell-\bar Y_\ell)^2$.
This factor describes the relationship between the rates of variance decreasing and kurtosis increasing, which can reveal the effectiveness of a method dealing with the high-kurtosis phenomenon. For the original MLMC and MLQMC methods, the KVF is very close to 1 for mildly large levels as explained in \cref{eq:orik}. Therefore, if a method can achieve a smaller KVF, then with a same $V_\ell$, the kurtosis is not so large, which means the high-kurtosis phenomenon is relieved and the estimation of $V_\ell$ can be more accurate.

\section{Numerical study}\label{sec:numerical}
In the numerical study we consider a portfolio that consists of European options written on $d$ stocks whose price dynamics follow the Black-Scholes model. For simplicity, we assume that the stock returns are the same, denoted by $\mu$, while risk-free interest rate is $\mu_0$. Price dynamics of the stocks $\bm S_t=(S^1_t,\dots,S^d_t)$ evolve according to
\begin{displaymath}
\frac{d S^i_t}{S_t^i}=\mu' d t+\sum_{j=1}^d\sigma_{ij}d W^i_t,\ i=1,\dots,d,
\end{displaymath}
where $\mu'=\mu$ under the real-world probability measure, and $\mu'=\mu_0$ under the risk-neutral probability measure. Here, $\bm W_t=(W^1_t,\dots,W^d_t)$ is a $d$-dimensional standard Brownian motion which represents $d$ risk factors in the model. We have,
\begin{displaymath}
S^i_t=S^i_0\exp\{(\mu'-\frac{1}{2}\sum_{j=1}^d\sigma^2_{ij})t+\sum_{j=1}^d\sigma_{ij}W_t^j\},\ i=1,\dots,d,
\end{displaymath}
where $\bm S_0=(S^1_0,\dots,S^d_0)$ are the initial prices of stocks.

We assume that the maturities of all the European options in the portfolio are the same, denoted by $T$. We want to measure the portfolio risk at a future time $\tau$ ($\tau<T$). In the simulation, we first simulate the random variable $\omega=\bm S_\tau=(S^1_\tau,\dots,S^d_\tau)$ under real-world probability measure, which denotes the prices of stocks at the risk horizon $\tau$ as the outer sample. And then simulate $\bm S_T=(S^1_T,\dots,S^d_T)$ under the risk-neutral probability measure which is the prices of stocks at maturity $T$ given $\omega$ as the inner samples. Denote by $V_0=\sum_{i=1}^d v^i_0$ the initial value of the portfolio, where each $v^i_0$ is known by the Black-Scholes formula \cite{Hull2017}. Then the portfolio value change is
\begin{displaymath}
g(\omega) := V_0-\mbe[V_T(\bm S_T)|\omega],
\end{displaymath}
where $V_T(\bm S_T)$ is the discounted payoff of the portfolio at time $T$ which is a known function of $\bm S_T$.

Our target is to estimate the loss probability $\theta=\mbp[g(\omega)>c]$ for a given threshold $c$.
We take $m_\ell=32\times 2^\ell$ in all experiments and $k_0=8$ for the smoothed methods. We compare MLMC, MLQMC, and smoothed MLQMC (denoted by SMLQMC) in each example.

\subsection{Single asset}
Firstly, we consider a portfolio consists of a single put option, i.e., $d=1$. This example was studied in \cite{broa:2011}.

In the simulation, the outer random variable is generated
according to
\begin{displaymath}
\omega = S_\tau=S_0\exp\{(\mu-\sigma^2/2)\tau+\sigma\sqrt{\tau}Z\},
\end{displaymath}
where the real-valued risk factor $Z$ is a standard normal random variable. The portfolio value change is
\begin{displaymath}
g(\omega) = v_0-\mbe[e^{-\mu_0(T-\tau)}(K-S_T(\omega,W))^+|\omega],
\end{displaymath}
where the expectation is taken over the random variable
$W$, which is a standard normal random variable independent with $Z$, and $S_T(\omega,W)$ is given by
\begin{displaymath}
S_T(\omega,W)= \omega\exp\{(\mu_0-\sigma^2/2)(T-\tau)+\sigma\sqrt{T-\tau}W\}.
\end{displaymath}

Now let $X = v_0-e^{-\mu_0(T-\tau)}(K-S_T(\omega,W))^+$. For $\omega=y$, $X$ can be generated via
\begin{displaymath}
X(y) = \psi(u;y)= v_0-e^{-\mu_0(T-\tau)}(K-y\exp\{(\mu_0-\sigma^2/2)(T-\tau)+\sigma\sqrt{T-\tau}\Phi^{-1}(u)\})^+.
\end{displaymath}
Note that $\psi(u;y)$ is a decreasing function with respect to $u$. So the variation of Hardy and Krause for $\psi(u;y)$ can be computed easily, i.e., for any $y$,
\begin{displaymath}
V_{\mathrm{HK}}(\psi(\cdot;y))=e^{-\mu_0(T-\tau)}K.
\end{displaymath}
As a result, when using scrambled $(t,1)$-sequence in base $b=2$ in the inner simulation, we have
\begin{displaymath}
\var{\hat g_m(y)}\le \frac{C}{m^2},
\end{displaymath}
where $C$ is a constant that does not depend on $y$, and the sample size has the form $m=2^\ell$. \cref{assm1} is thus satisfied with $\eta=2$ and $\sigma(y)\equiv C.$

Note that $g(y)$ is a strictly increasing continuous function. So the cumulative distribution function of $g(\omega)$ can be computed easily, namely,
\begin{align*}
\mbp[g(\omega)\le t] &= \mbp[\omega\le g^{-1}(t)]\\
&=\mbp[S_0\exp\{(\mu-\sigma^2/2)\tau+\sigma\sqrt{\tau}Z\}\le g^{-1}(t)]\\
&=\Phi([\log(g^{-1}(t)/S_0)-(\mu-\sigma^2/2)\tau]/(\sigma\sqrt{\tau})).
\end{align*}
The density of $g(\omega)$ is then given by
\begin{displaymath}
f(t) =\frac{\phi([\log(g^{-1}(t)/S_0)-(\mu-\sigma^2/2)\tau]/(\sigma\sqrt{\tau}))}{\sigma\sqrt{\tau}g^{-1}(t)g'(g^{-1}(t))}.
\end{displaymath}
It is easy to see that $f(0)<\infty$ so that \cref{assum:var} is satisfied. The MLQMC thus gives a complexity of $O(\epsilon^{-2}(\log \epsilon)^2)$ by \cref{eq:complexity}.

\begin{figure}[htbp]
  \centering
  \subfigure[]{\label{ex2a}\includegraphics[width=4.26cm]{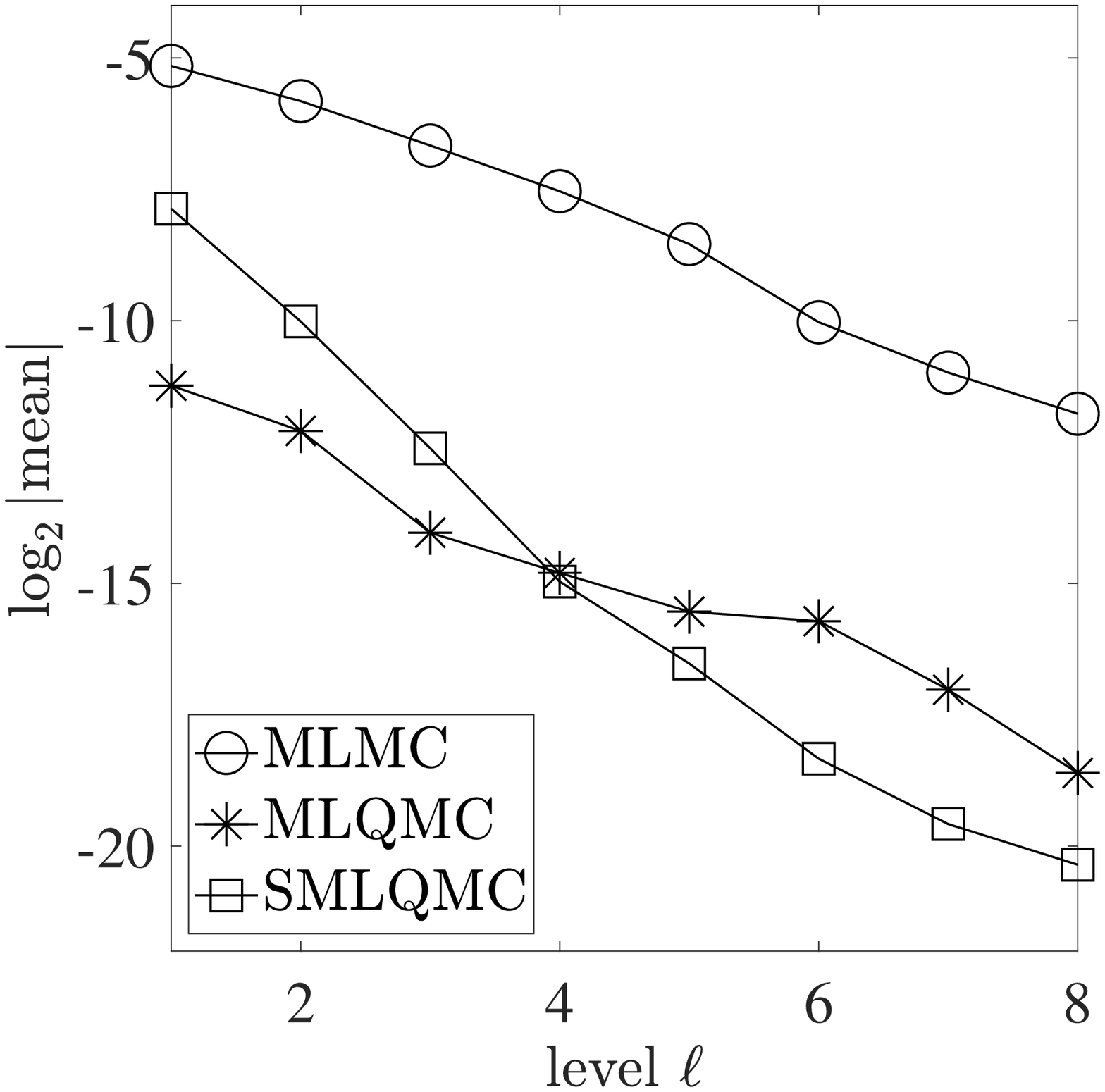}}
  \subfigure[]{\label{ex2b}\includegraphics[width=4.26cm]{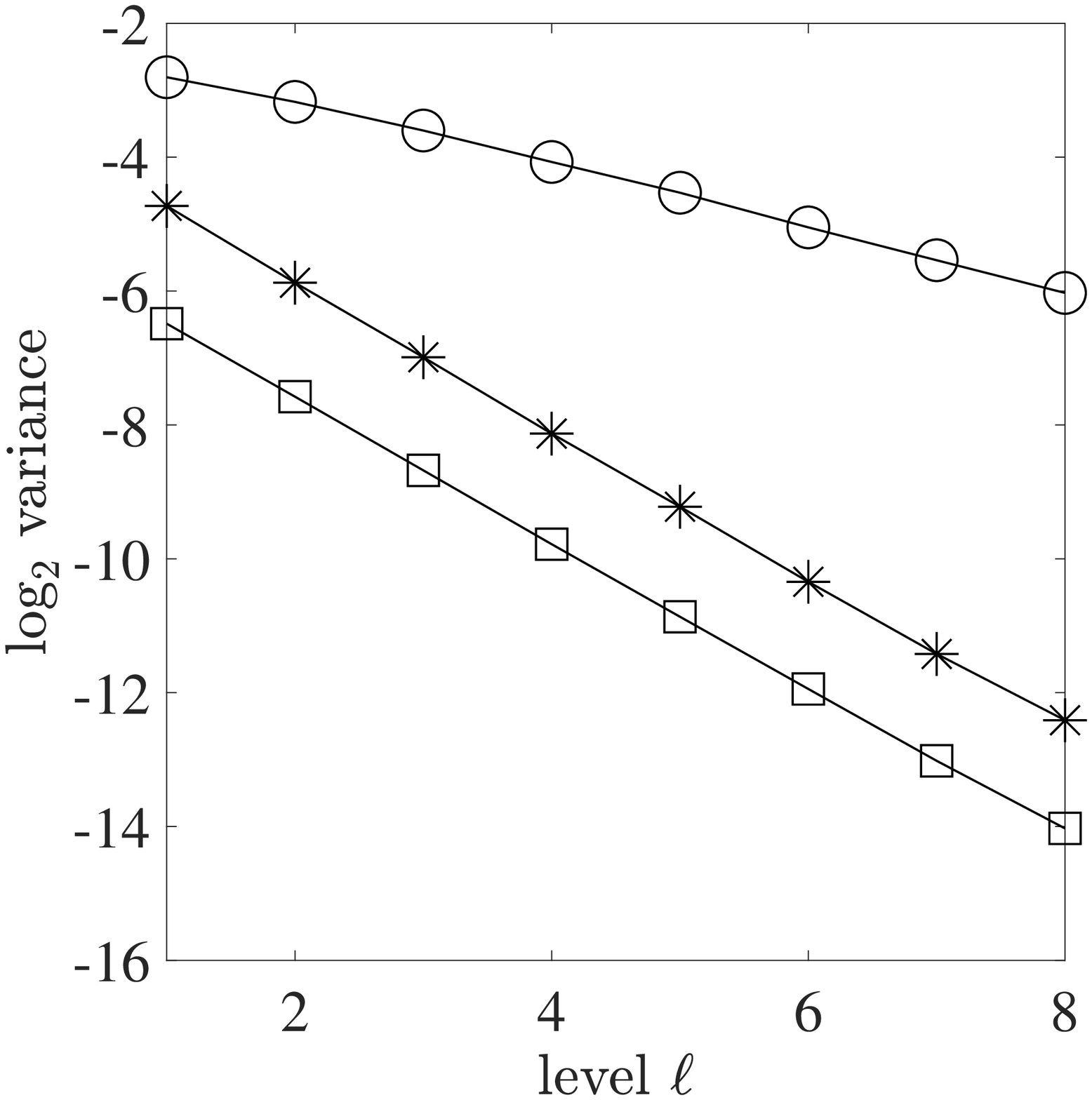}}
  \subfigure[]{\label{ex2c}\includegraphics[width=4.26cm]{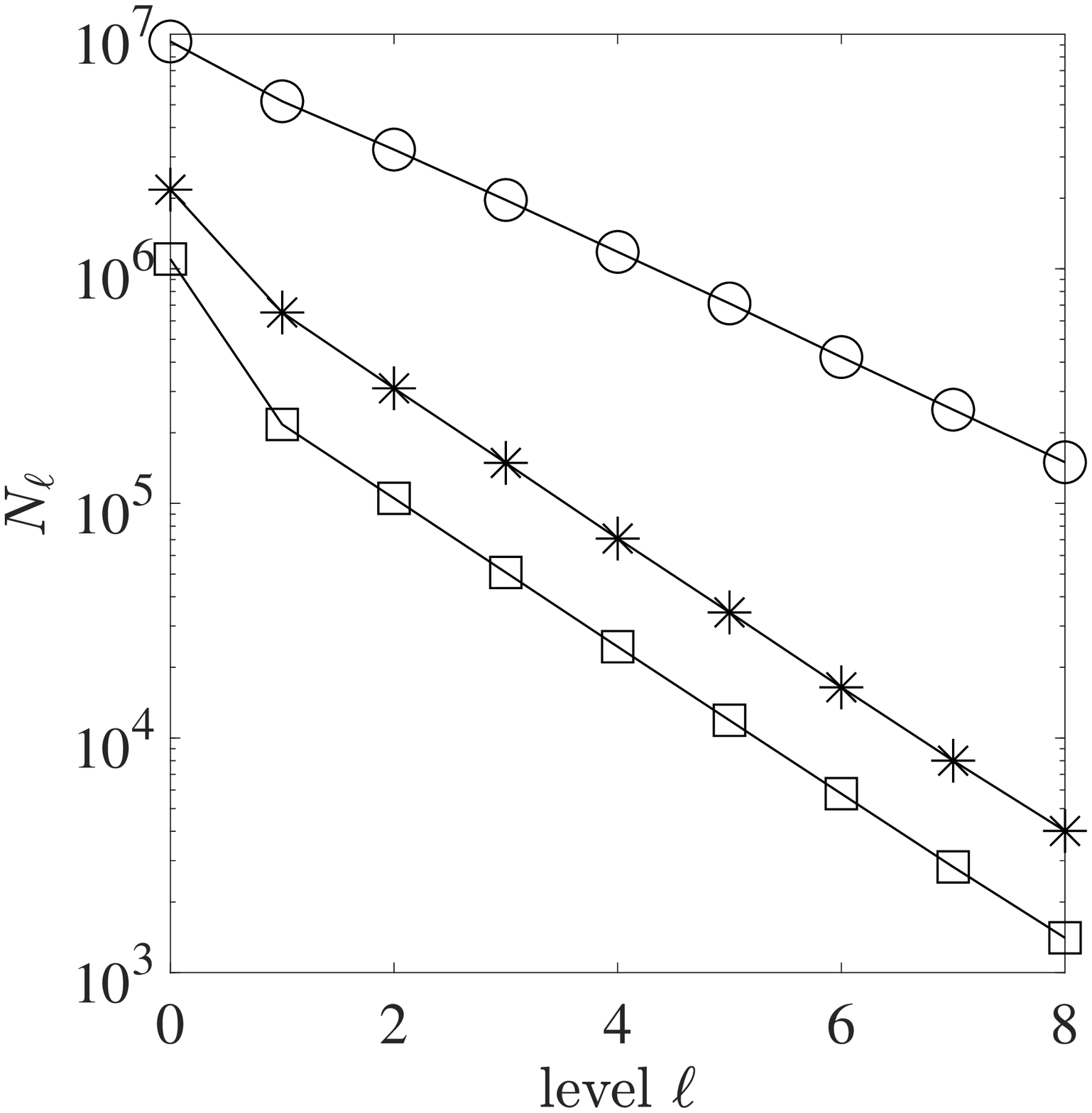}}
  \subfigure[]{\label{ex2d}\includegraphics[width=4.26cm]{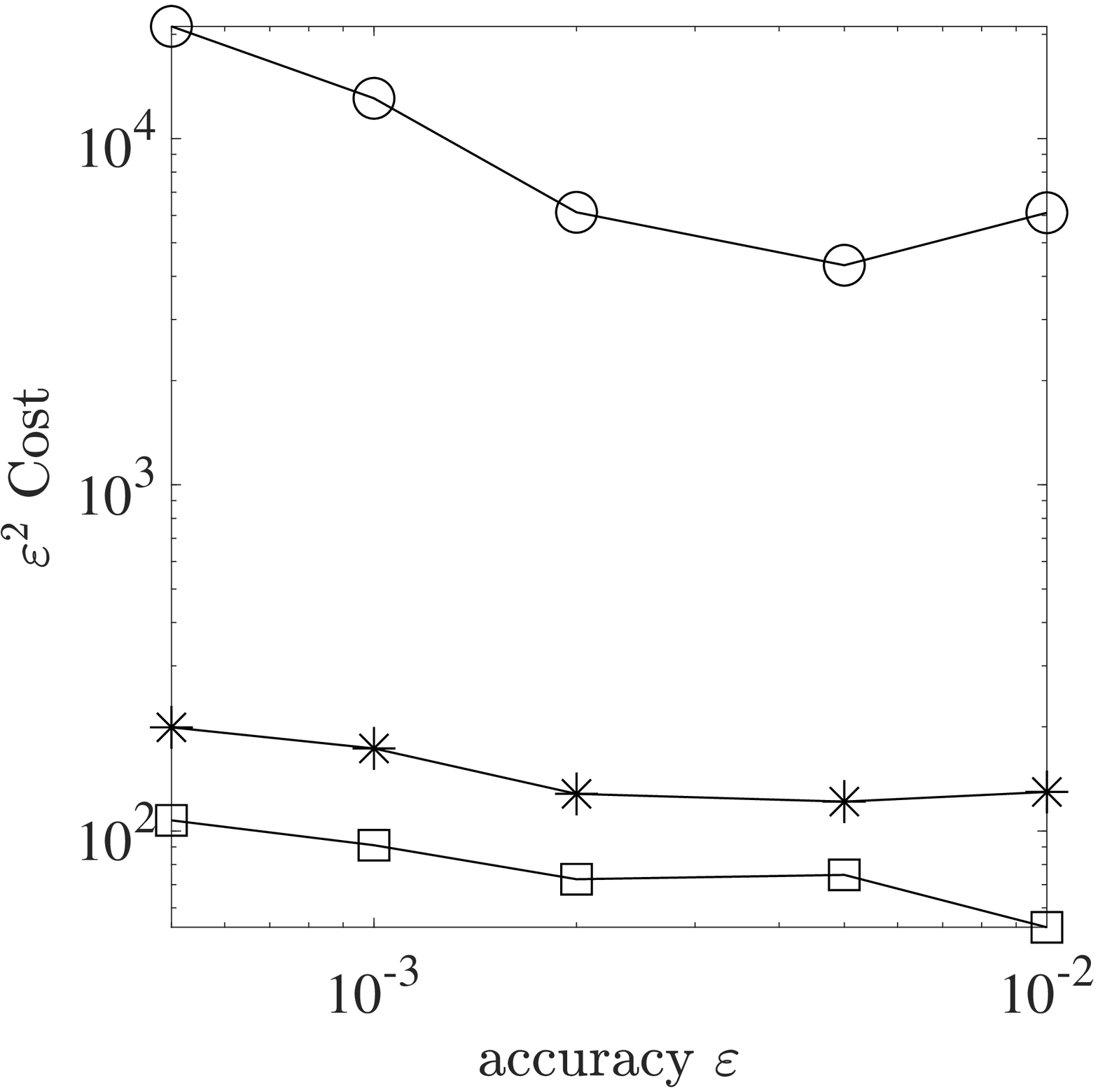}}
  \subfigure[]{\label{ex2e}\includegraphics[width=4.26cm]{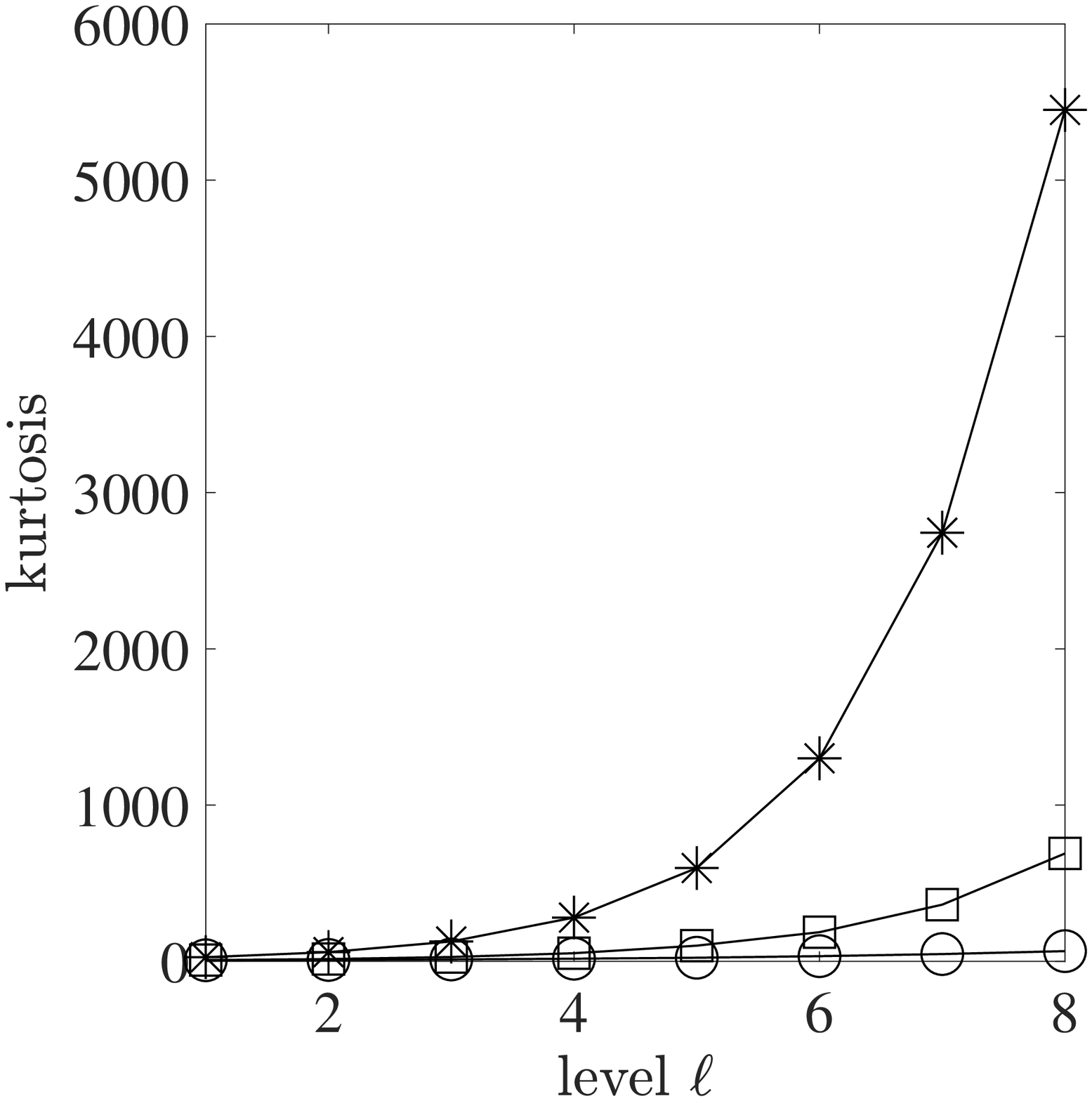}}
  \subfigure[]{\label{ex2f}\includegraphics[width=4.26cm]{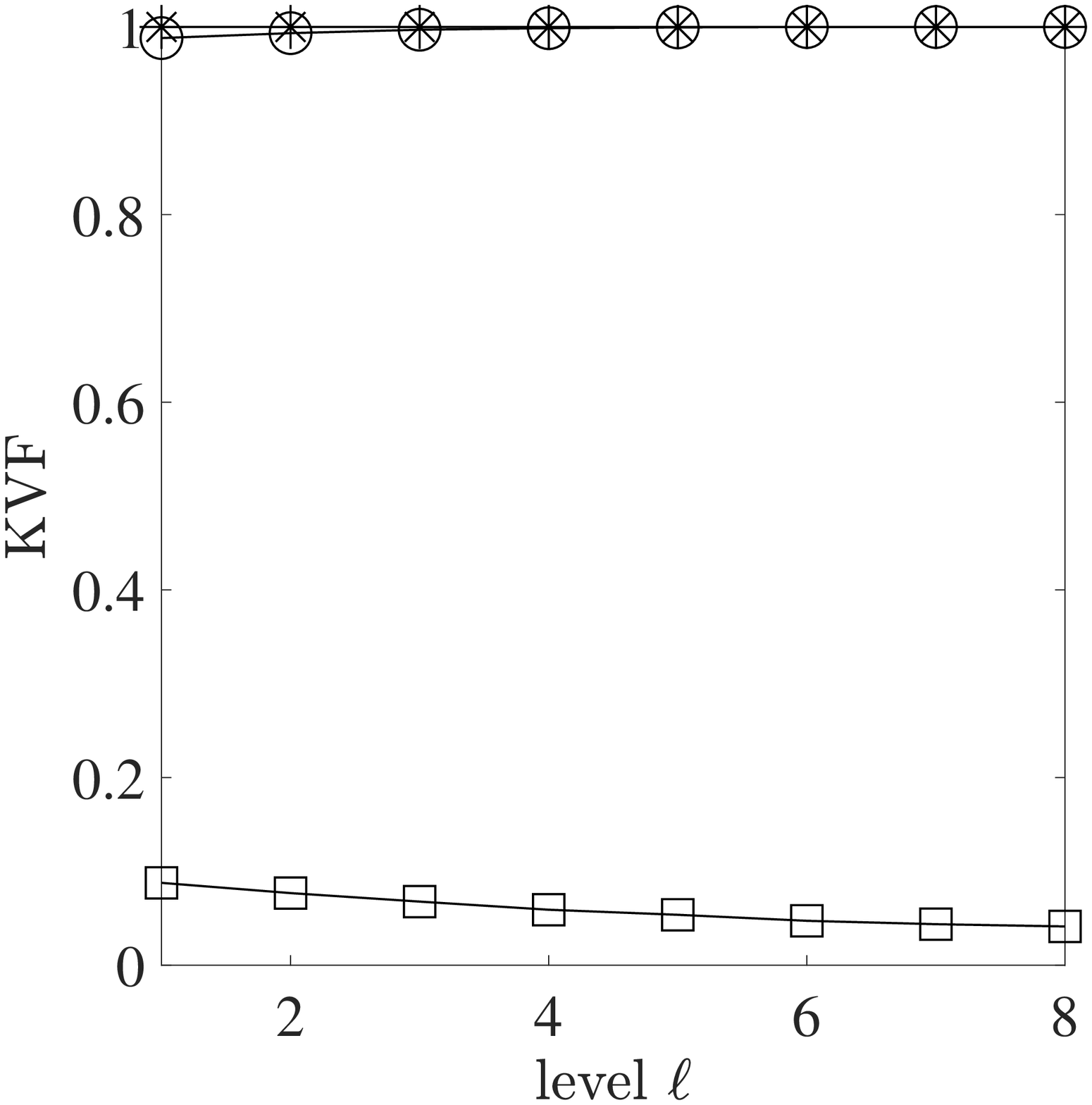}}
  \caption{Comparison of MLMC, MLQMC and SMLQMC for the single asset case.}
  \label{fig:ex1comparison}
\end{figure}

The parameters in this case are set as follows: $S_0=100$, $\mu = 8\%$, $\mu_0=3\%$, the annualized volatility $\sigma=20\%$, the strike of the put option $K = 95$, the maturity $T=0.25$ years (i.e., three months), and the risk horizon $\tau = 1/52$ years (i.e., one
week). With these parameters, the initial value of the
put option is $v_0=1.669$ by the Black--Scholes formula.
The inner expectation $g(\omega)$ can be computed explicitly by Black-Scholes formula. We choose $c=0.476887$ so that the loss probability $\theta=0.3$. For this one-dimensional problem, we take $r=2$ in the smoothed MLQMC. The results on testing the weak convergence and the strong convergence in \cref{fig:ex1comparison} are based on 500,000 outer samples in each level.

\cref{ex2a} displays the weak convergence of the three methods, from which we can see that both crude MLQMC and smoothed MLQMC method result in larger $\alpha$ than the MLMC method. This implies that QMC methods enjoy a better weak convergence and converge to the true value faster.

\cref{ex2b} shows the comparison of strong convergence, which plays a central role in MLMC methods. It is obvious that QMC methods work much better than MLMC, yielding the slopes of the lines lower than $-1$. Thus, QMC methods make the constant $\beta$ exceed $\gamma=1$, which means the variance $V_\ell$ decreases faster than the cost $C_\ell$ increases. And the smoothed MLQMC method shares the same rate with the crude MLQMC method but achieves an even smaller variance. As a result, the costs in each level are reduced accordingly, which is confirmed by \cref{ex2c}.

As we stated earlier, QMC methods are able to accelerate the strong convergence, leading to  $\beta>\gamma=1$ in \cref{thm:mlmc}. Then the total cost falls in the third regime in \cref{thm:mlmc}, which is the optimal complexity bound $O(\epsilon^{-2})$ for MLMC methods. The numerical results support this in \cref{ex2d}. Both QMC methods are of the best complexity, while the smoothed MLQMC method reduces further the computational burden.

\cref{ex2e} displays the tendency of kurtosis. As we explained, the kurtosis increases with the variance decreasing as $\ell$ gets larger. This phenomenon is inevitable due to the structure of the indicator function. But combined with the smoothed method, the increasing of kurtosis slows down obviously, which means the smoothed MLQMC method takes effect on high-kurtosis phenomenon.

More visually, in \cref{ex2f}, we can see the KVFs of the MLMC and the crude MLQMC methods are close to 1 in all the cases, while the smoothed MLQMC method yields much smaller KVF. This implies that the smoothed method slows down the rate of kurtosis increasing. In this sense, we can see the effect of smoothed MLQMC on overcoming the high-kurtosis phenomenon and catastrophic coupling.

\subsection{Multiple assets}
Now we consider a portfolio consists of $d$ European call options, which was studied in \cite{Hong2017}. Given the outer sample $\omega=\bm S_\tau=(S^1_\tau,\dots,S^d_\tau)$ which denotes the prices of stocks at the risk horizon $\tau$ under real-world measure, and $\bm K=(K^1,\dots,K^d)$, the strike prices for call options, the portfolio value change is
\begin{displaymath}
g(\omega) = V_0-\mbe[e^{-\mu_0(T-\tau)}\sum_{i=1}^d(S^i_T(\omega^i,\bm W)-K^i)^+|\omega],
\end{displaymath}
where $\omega^i$ is the $i$th element of the vector $\omega$ and the expectation is taken over the random variable
$\bm W=(W^1,\dots,W^d)\sim N(\bm 0,\bm I_d)$, and samples of
$$S^i_T(\omega^i,\bm W)
=\omega^i\exp\{(\mu_0-\frac{1}{2}\sum_{j=1}^d\sigma^2_{ij})(T-\tau)+\sum_{j=1}^d\sigma_{ij}\sqrt{T-\tau}W^j\}$$
are simulated under the risk-neutral measure.

To handle the effect of high dimensionality, we combine the gradient PCA (GPCA) method  proposed by Xiao and Wang \cite{Xiao2017} within SMLQMC, which is abbreviated as GMLQMC hereafter. GPCA is a general method to reduce the effective dimension of functions for improving the efficiency of QMC method. Here we apply GPCA in the inner simulation. We also examine the antithetic sampling of GMLQMC (abbreviated as AMLQMC). As we know, using an antithetic estimator does not change the variance convergence rate in the crude MLMC because of the indicator function. It is of interest to test whether antithetic sampling can get more benefit in the smoothed method.

The parameters in our experiments are set as follows: $S^1_0=\dots=S^d_0=100$, $\mu = 8\%$, $\mu_0=5\%$, the strikes $K^1=\dots=K^d=95$, the maturity $T=0.1$, and the risk horizon $\tau =0.02$. Without loss of generality, we let $\Sigma=(\sigma_{ij})$ be a sub-triangular matrix satisfying $C=\Sigma\Sigma^T$ which corresponds with Cholesky decomposition of $C$, where $C_{ij}=0.3\cdot0.98^{|i-j|}$. The different decompositions of $C$ do not influence the efficiency of MC method, but take effect under QMC scheme, while the Cholesky decomposition is a common and standard method in QMC simulation. The threshold $c=20\%V_0$. We take $r=\sqrt2$ in the smoothed MLQMC methods.

\begin{figure}[htbp]
  \centering
  \subfigure[$d=4$]{\includegraphics[width=4.26cm]{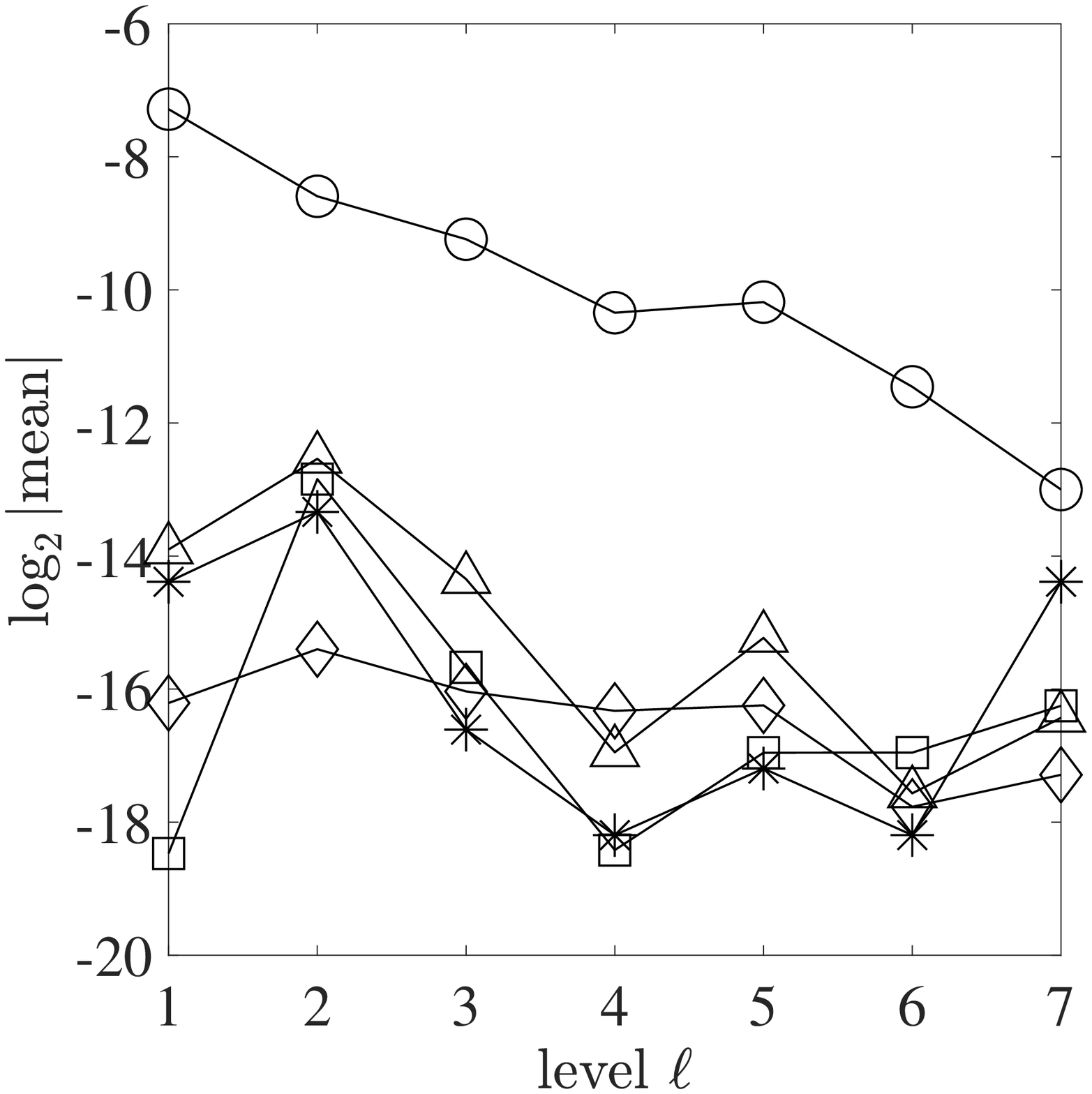}}
  \subfigure[$d=16$]{\includegraphics[width=4.26cm]{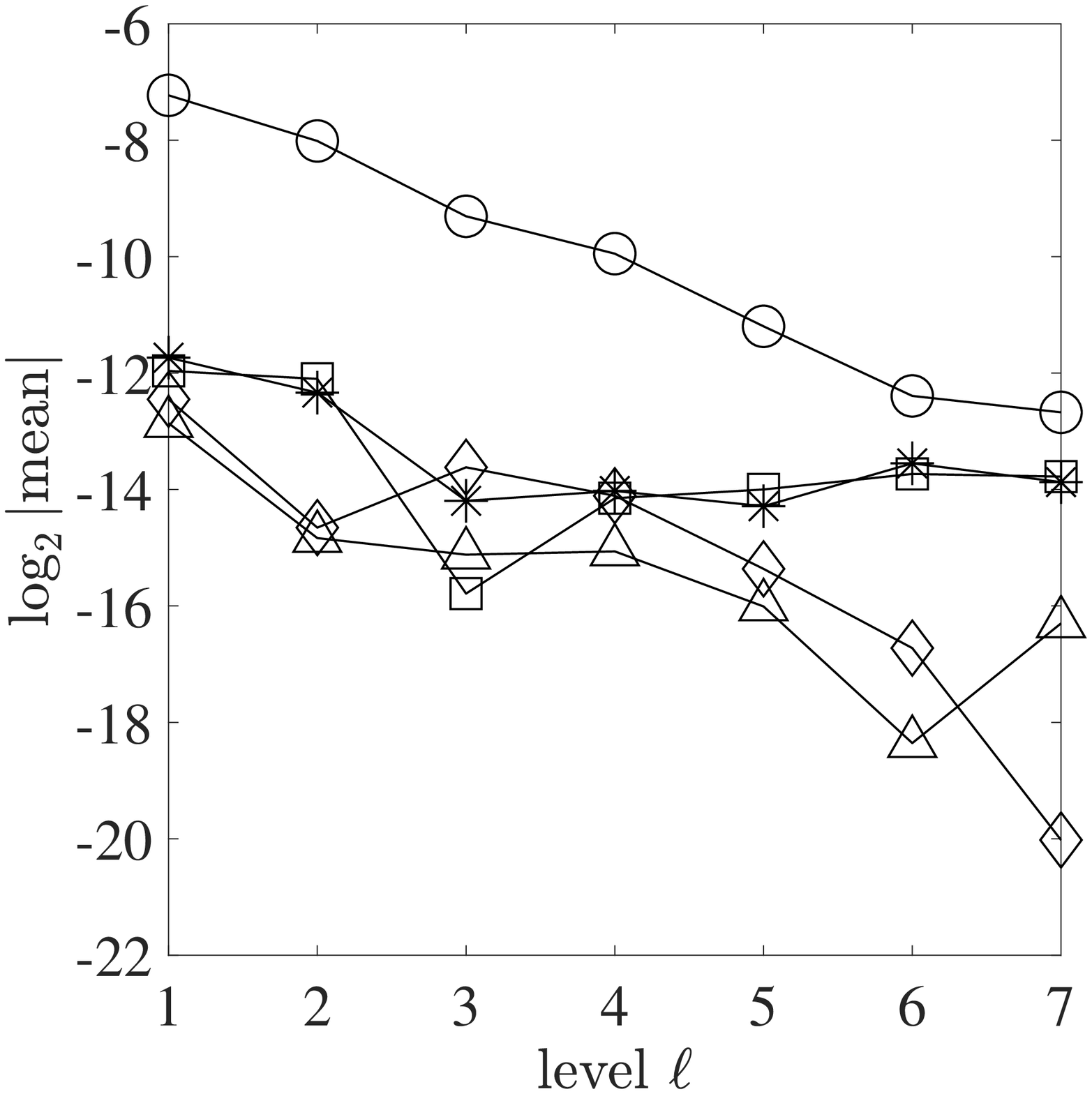}}
  \subfigure[$d=32$]{\includegraphics[width=4.26cm]{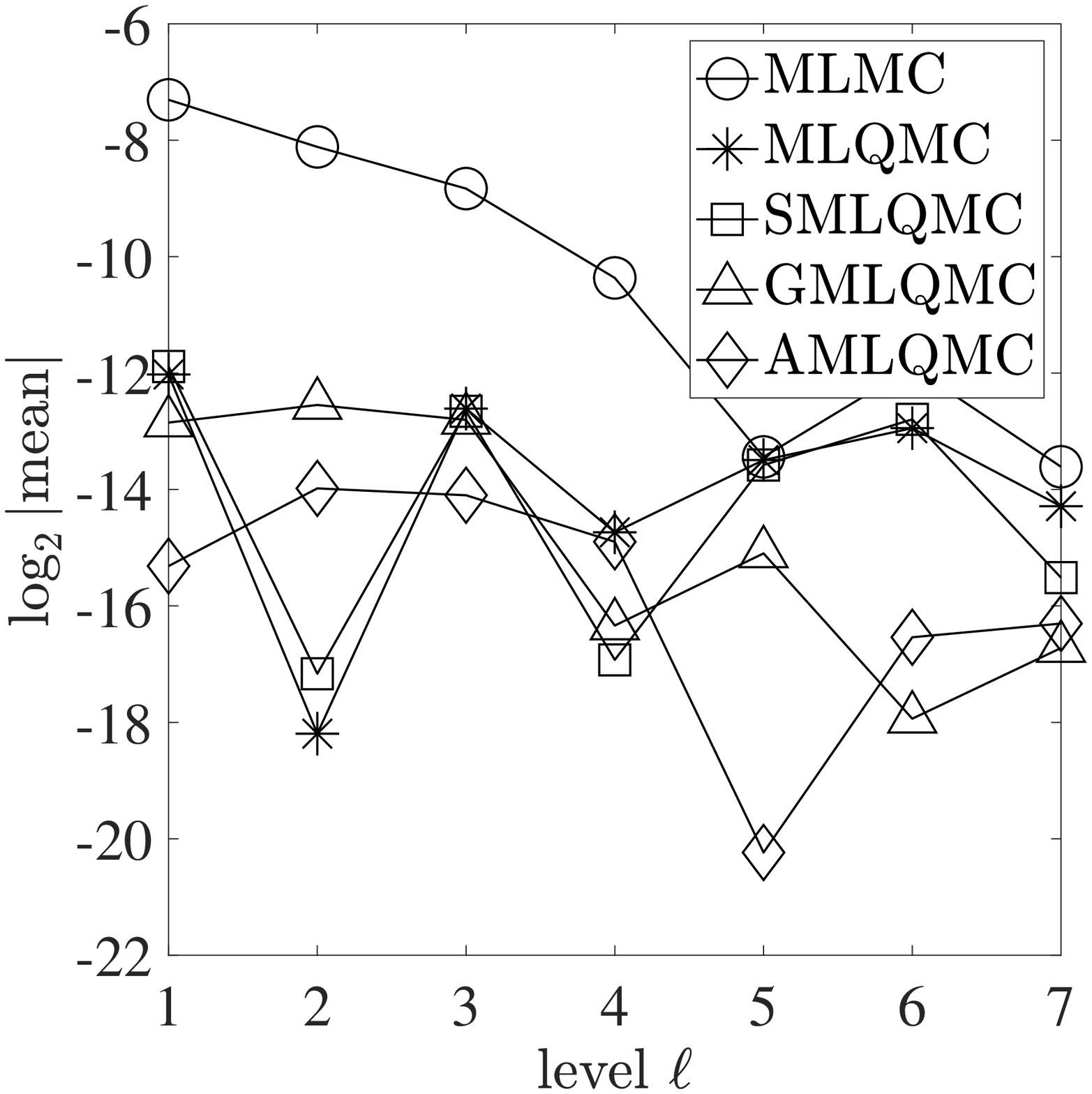}}
  \caption{Estimations of $|\mbe{Y_\ell}|$ or $|\mbe{\tilde Y_\ell}|$.}
  \label{fig:expection}
\end{figure}

Firstly, we perform the convergence tests for estimating $\alpha$ and $\beta$ in \cref{thm:mlmc}. The results are based on 300,000 outer samples in each level. \cref{fig:expection} shows the behaviors of the absolute values of the expectations of $Y_\ell$ for crude methods and $\tilde Y_\ell$ for smoothed methods. It can be seen that QMC methods have smaller values but with more volatility. This is caused by the catastrophic coupling mentioned in \cref{sec:review}, because there is only a tiny proportion of samples different from zero. The high sensitivity makes it difficult to predict the value of $\alpha$, even with such a size of outer samples.

\begin{figure}[htbp]
  \centering
  \includegraphics[width=4.26cm]{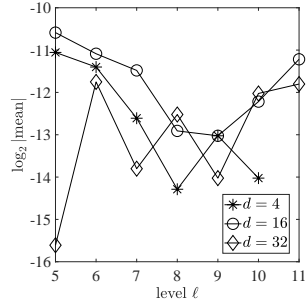}
  \caption{Estimations of $|\mbe{Y_\ell}|$ for MLMC.}
  \label{fig:MCtest}
\end{figure}

Not only MLQMC methods are faced with this difficulty, so the MLMC method is. \cref{fig:MCtest} shows the results of the MLMC method in deep levels. Under the same magnitude, both MLMC and MLQMC have large volatility. But it can be observed in \cref{fig:expection} that the rates of the MLQMC methods are not worse than the MLMC method and the value of $\alpha$ does not matter here actually for the strong convergence being good enough. Nevertheless, it should be noted that the smoothed MLQMC methods are affected less by the catastrophic coupling and look more stable.

\begin{figure}[htbp]
  \centering
  \subfigure[$d=4$]{\includegraphics[width=4.26cm]{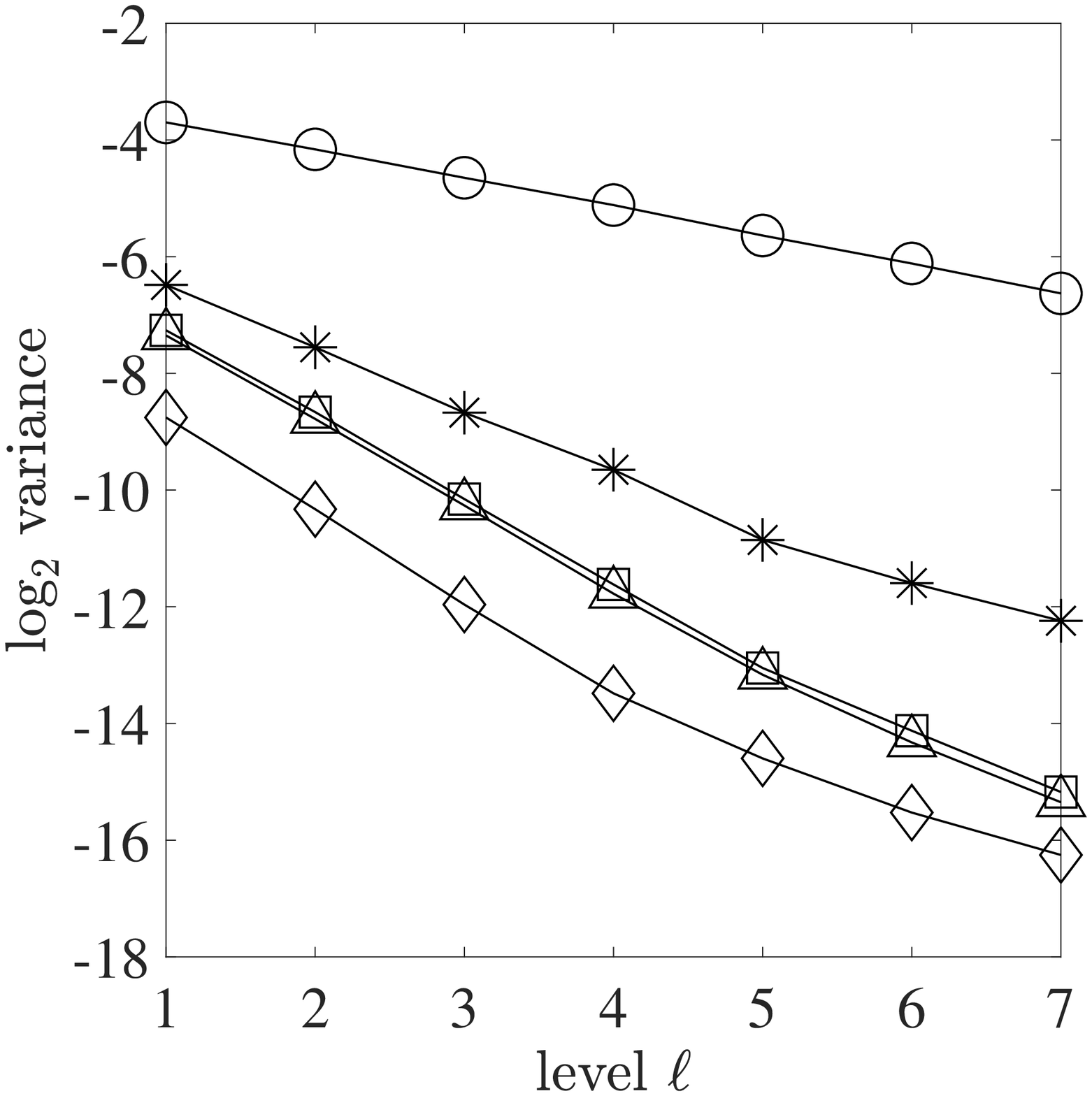}}
  \subfigure[$d=16$]{\includegraphics[width=4.26cm]{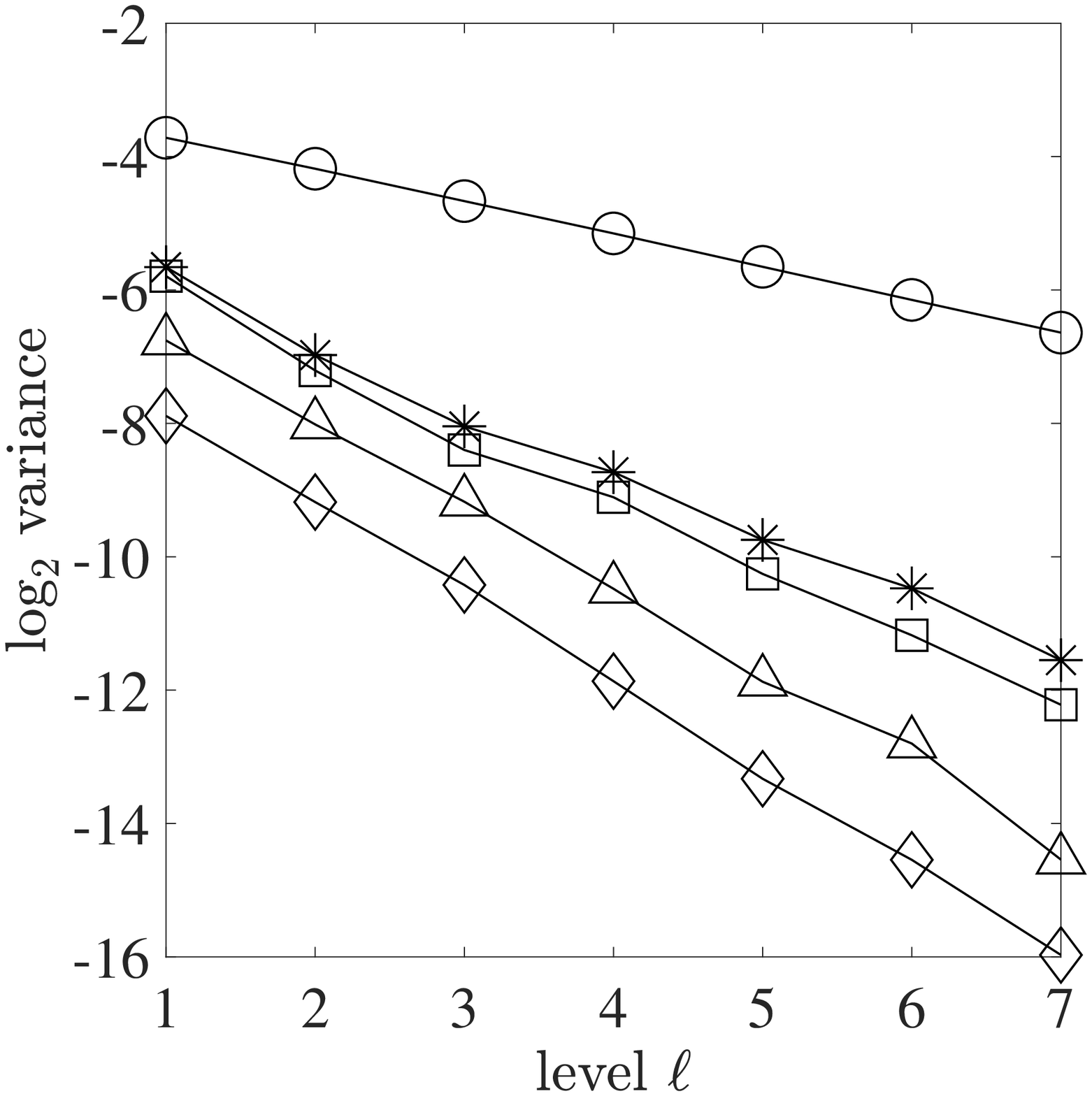}}
  \subfigure[$d=32$]{\includegraphics[width=4.26cm]{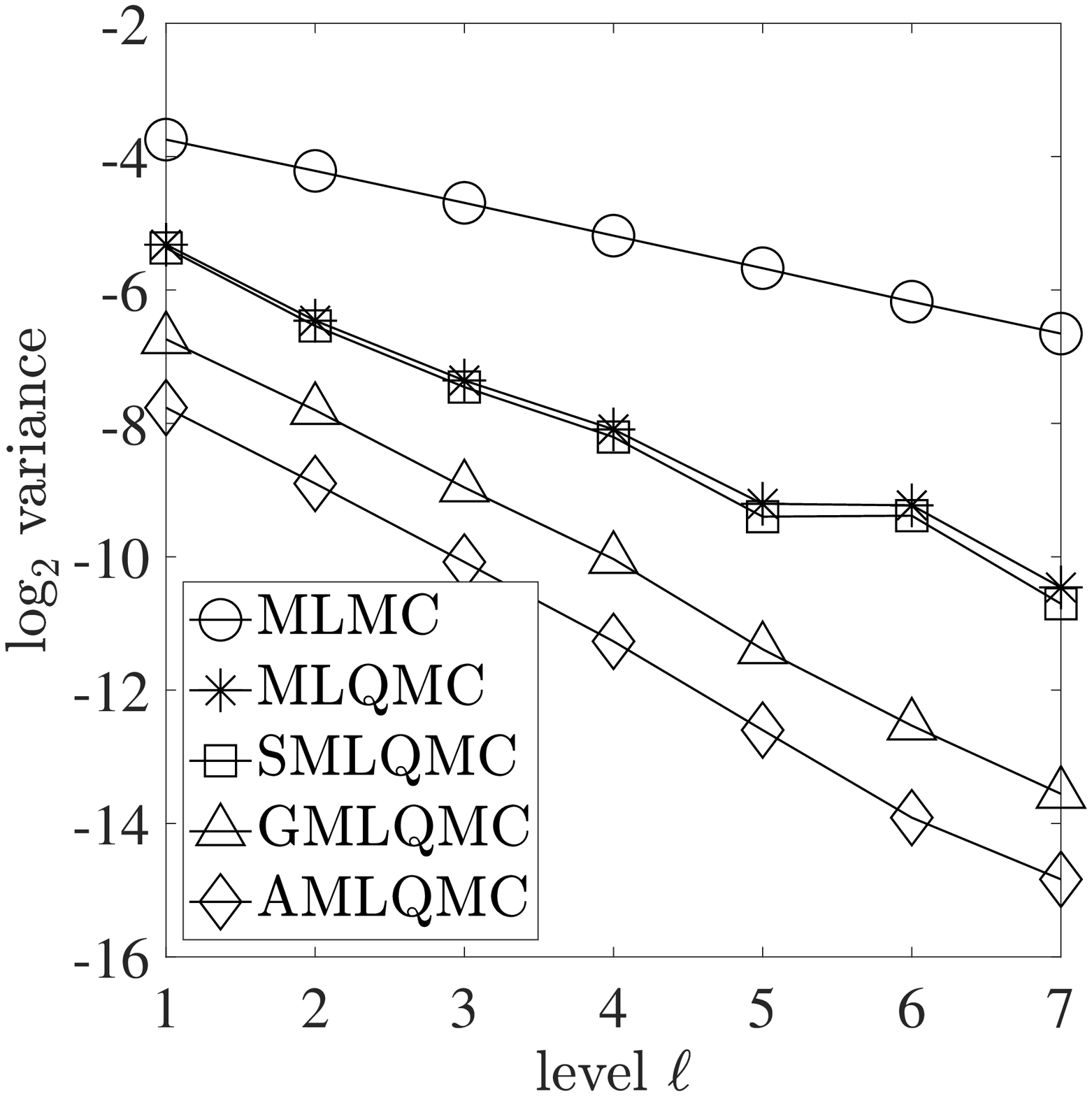}}
  \caption{Estimations of $\var{Y_\ell}$ or $\var{\tilde Y_\ell}$.}
  \label{fig:variance}
\end{figure}
\cref{fig:variance} shows the empirical variances of $Y_\ell$ and $\tilde Y_\ell$ for different levels, which can be used to predict the strong convergence rate $\beta$ by the usual linear regression. For the plain MLMC, we observe $\beta\approx 0.5$ for all dimensions. For the crude MLQMC method, the dimension $d$ has an impact on $\beta$, and $\beta\approx 1$ for moderately large $d=32$. When combined with GPCA method in the MLQMC method  (i.e., the GMLQMC method), we observe a larger $\beta$, usually $1.1$. When antithetic sampling method is applied in smoothed MLQMC method, $\beta$ is further improved to $1.2$. These insights suggest that antithetic sampling can benefit from the smooth coupling, achieving a larger $\beta$.
The strong convergence gets apparent improvement with QMC methods, in which $\beta=\gamma$ even $\beta>\gamma$. The total cost then is $O(\epsilon^{-2}(\log\epsilon)^2)$ even $O(\epsilon^{-2})$.

\begin{figure}[htbp]
  \centering
	\subfigure[$d=4$]{\includegraphics[width=4.26cm]{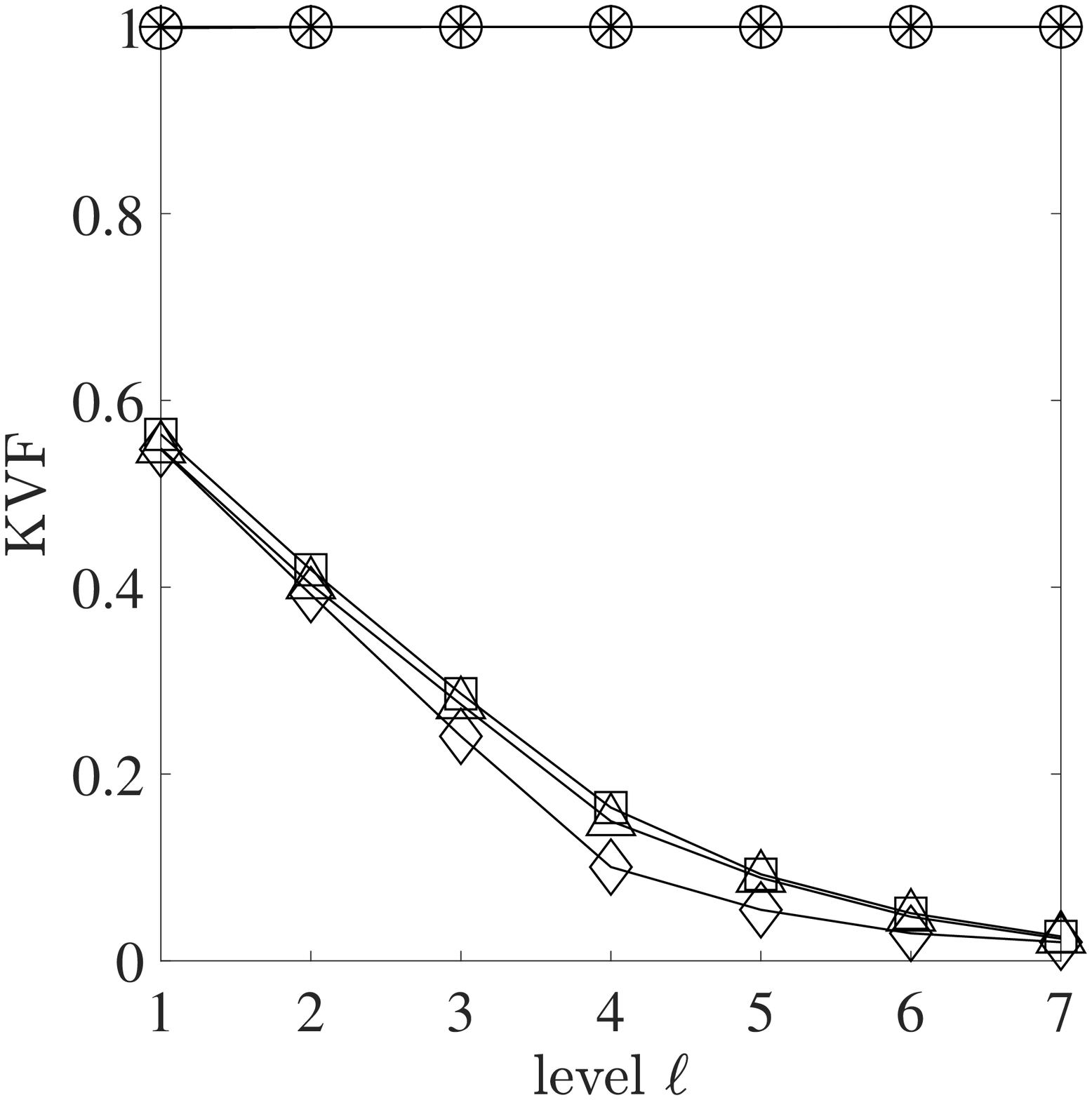}}
    \subfigure[$d=16$]{\includegraphics[width=4.26cm]{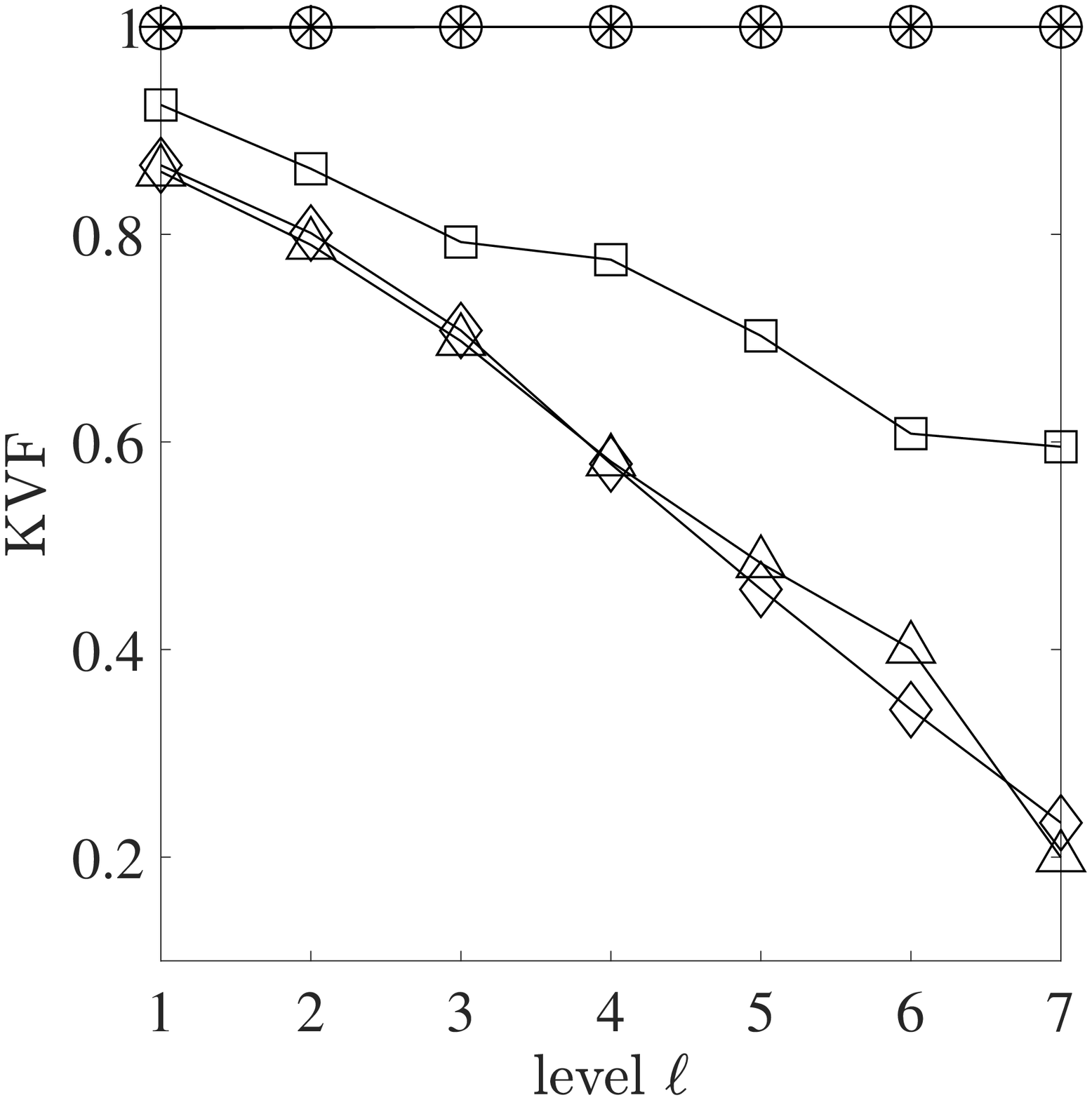}}
    \subfigure[$d=32$]{\includegraphics[width=4.26cm]{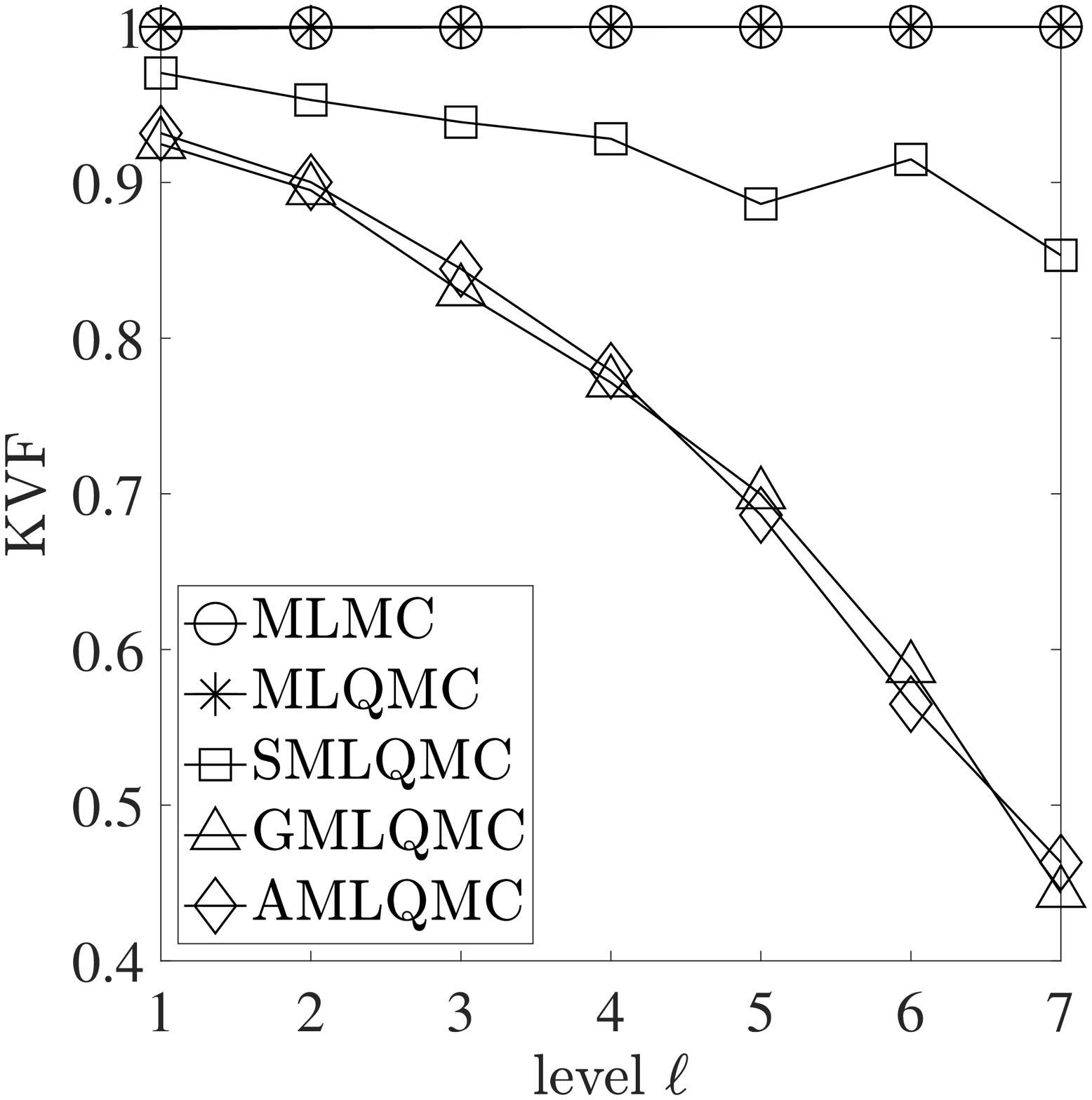}}
  \caption{Tests of KVF.}
\label{fig:KVF}
\end{figure}

Similar to the case of $d=1$, the kurtosis also becomes larger as the level $\ell$ increases. Here we focus on KVF to take both variance and kurtosis into consideration; see \cref{fig:KVF} for the results. It should be noticed that the curve of AMLQMC showed is 4 times KVF of AMLQMC actually, in order to balance the influence bought by the antithetic sampling. This is because the antithetic sampling can reduce the variance and kurtosis by 1/2 for the plain MLMC and crude MLQMC without changing the convergence rate, then it is fairer to compare 4 times KVF of the method with antithetic sampling. The MLMC and the MLQMC methods without the smooth coupling have KVFs close to 1. On the other hand, the smoothed methods achieve faster strong convergence and slower kurtosis increasing,  and so smaller values of KVF are observed. It can be seen that the KVF gets smaller in deeper levels, which means the smoothed MLQMC methods make the kurtosis increase slowly. In that sense, the smoothed MLQMC helps to overcome the catastrophic coupling and make the algorithm more efficient.

\begin{figure}[htbp]
  \centering
	\subfigure[$d=4$]{\includegraphics[width=4.26cm]{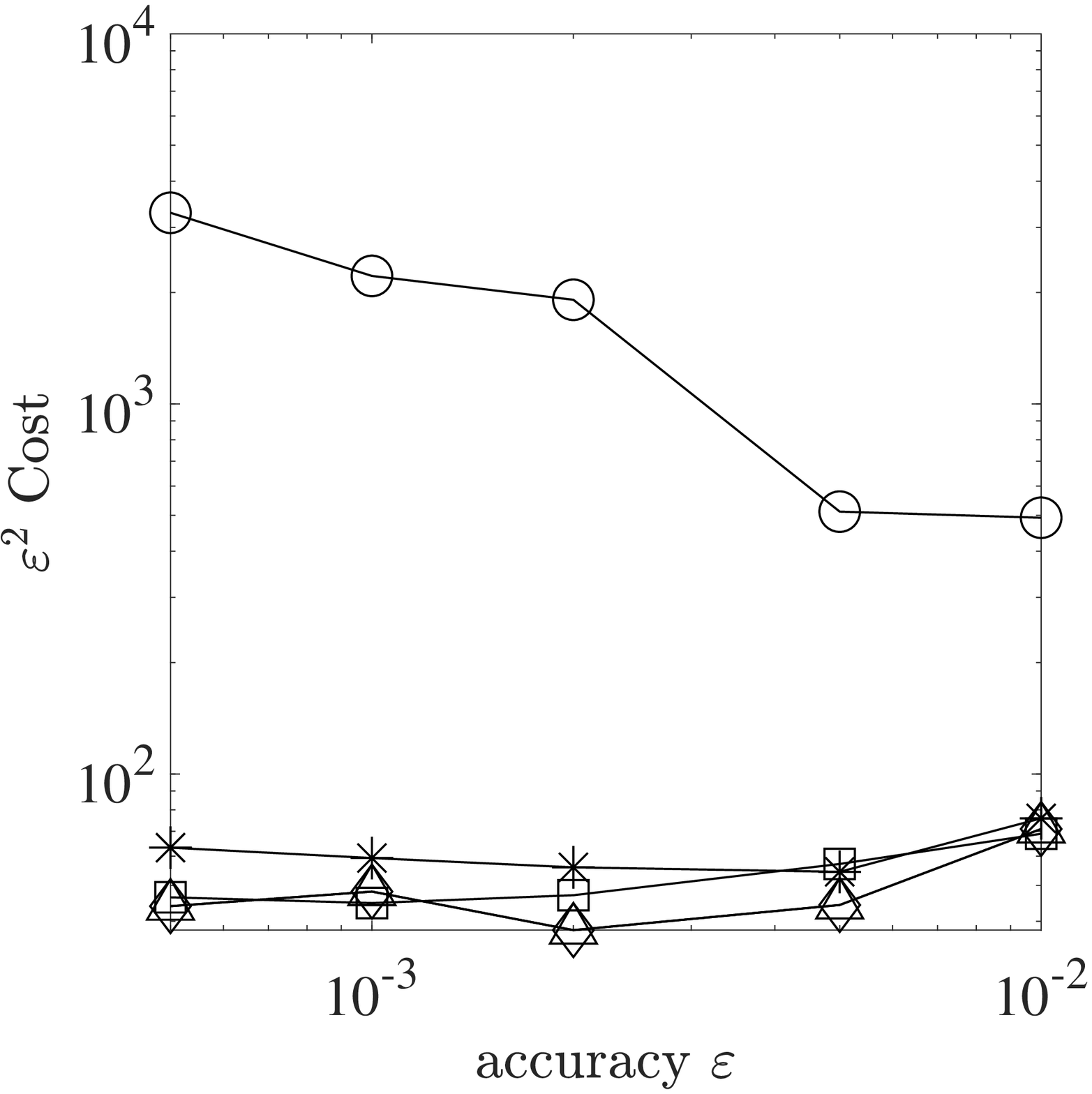}}
    \subfigure[$d=16$]{\includegraphics[width=4.26cm]{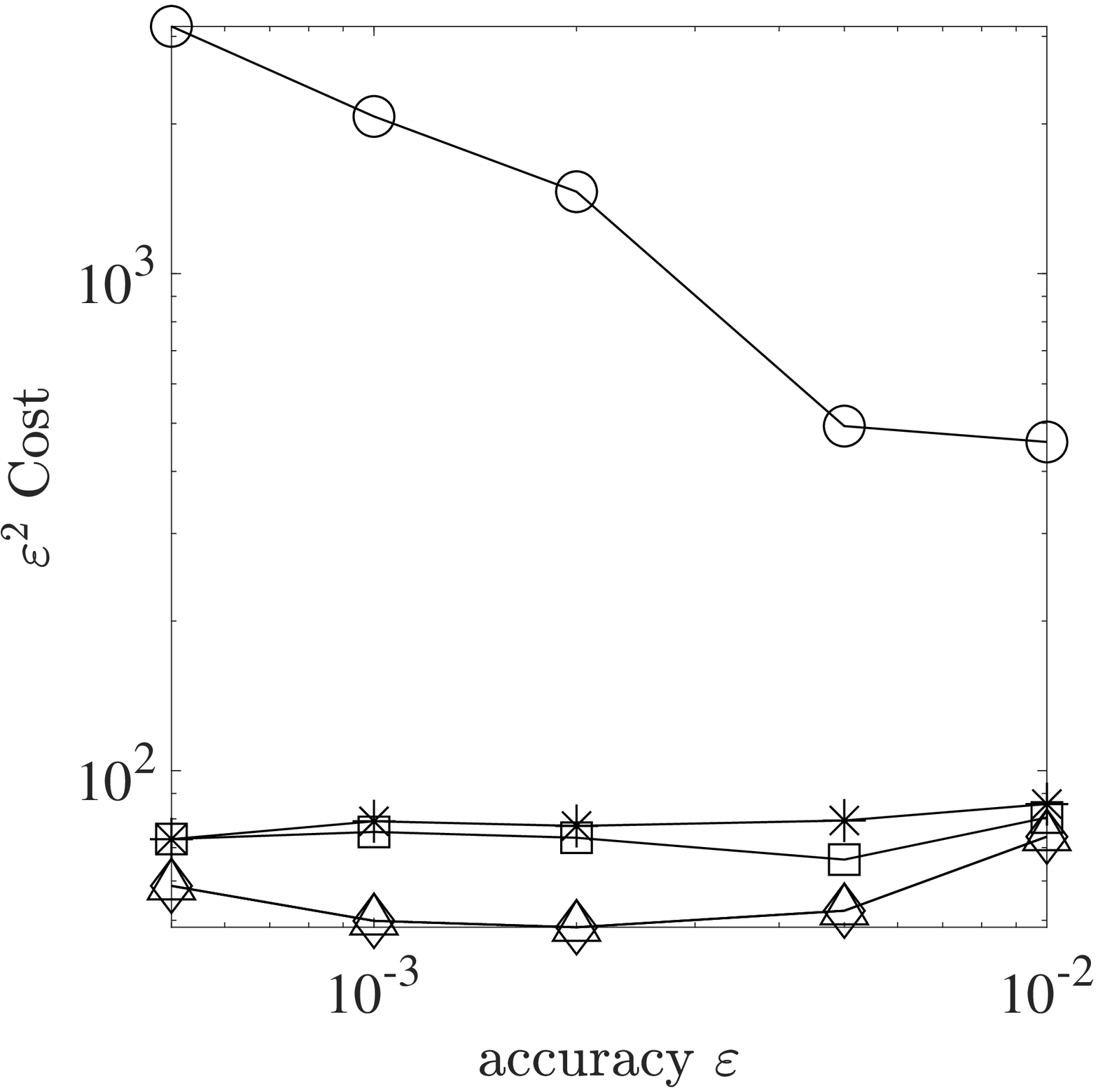}}
    \subfigure[$d=32$]{\includegraphics[width=4.26cm]{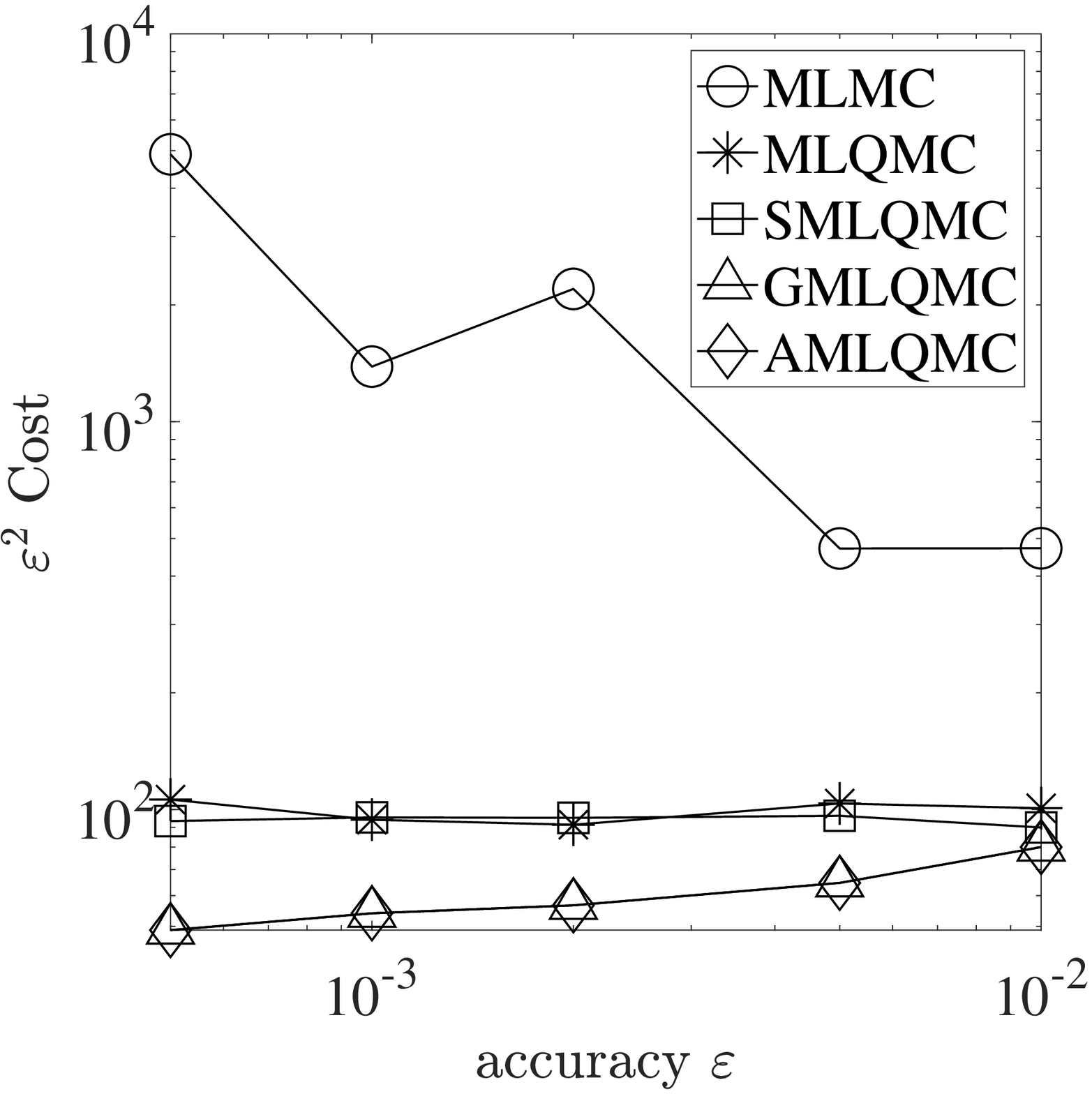}}
  \caption{Tests of total cost.}
\label{fig:totalcost}
\end{figure}

\cref{fig:totalcost} shows the total computation cost of each method. We can see that the complexity bound of MLQMC methods is $O(\epsilon^{-2})$ or $O(\epsilon^{-2}\log(\epsilon)^2)$ as expected. The smooth coupling indeed helps to reduce the total cost.

\begin{figure}[htbp]
  \centering
	\subfigure[$d=32$]{\includegraphics[width=4.26cm]{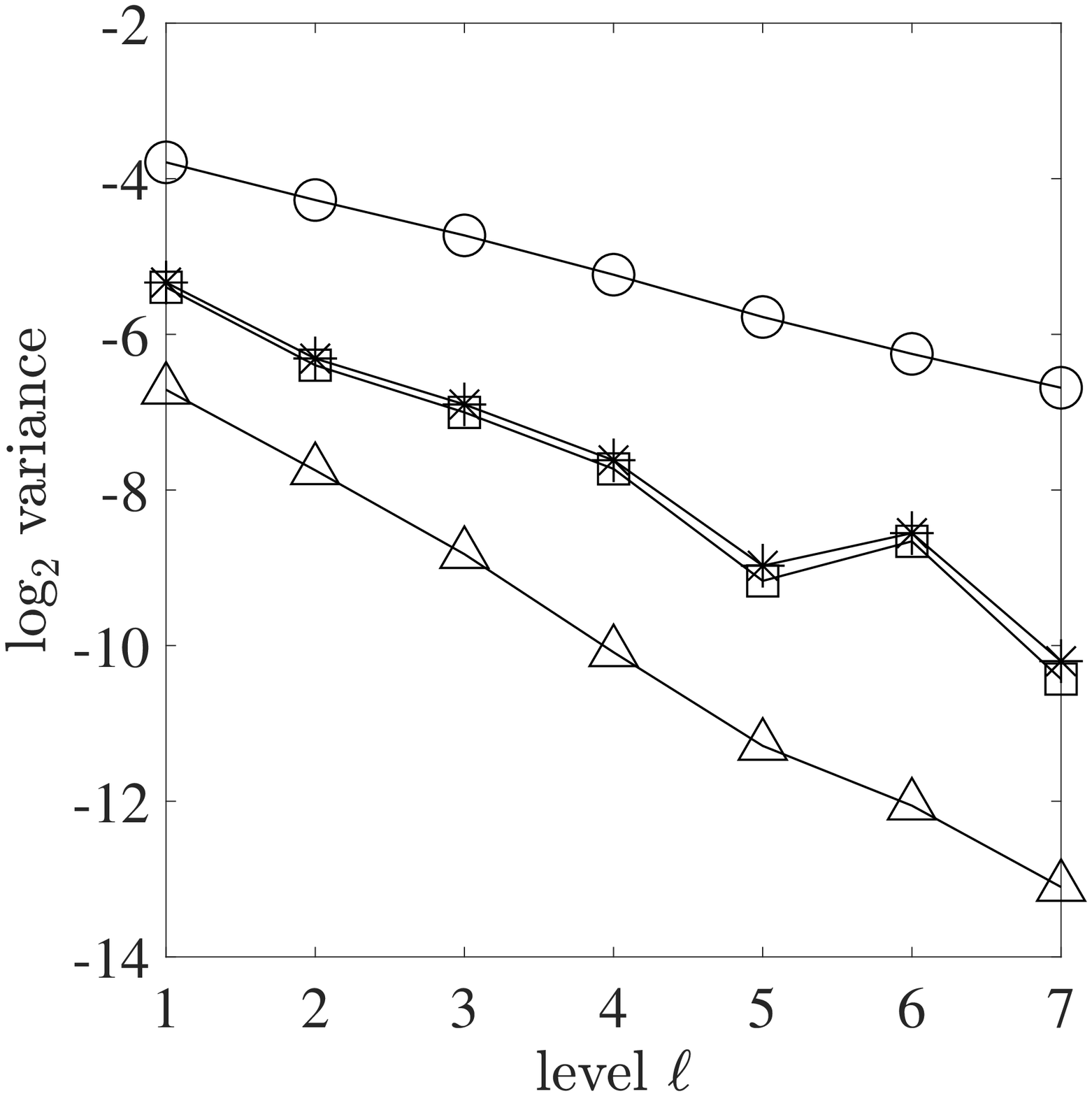}}
	\subfigure[$d=64$]{\includegraphics[width=4.26cm]{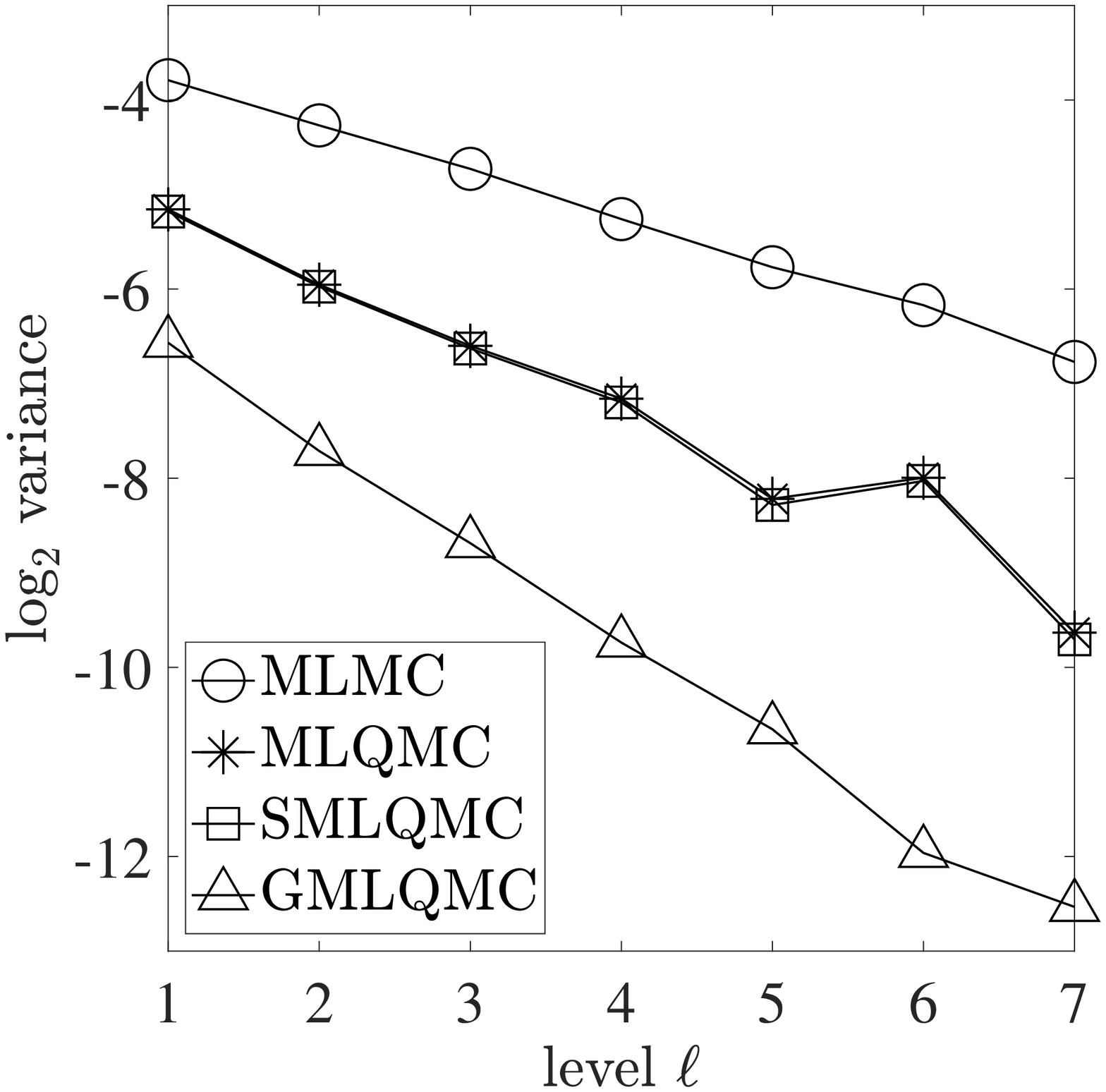}}
  \caption{Tests of dimension effects.}
\label{fig:dimcompare}
\end{figure}

As shown in \cref{fig:variance} and \cref{fig:totalcost}, the crude MLQMC method is good enough, almost attaining the optimal complexity, even for the large dimension $d=32$. The effect of the dimension $d$ on the performance of the crude MLQMC method seems minor. A possible explanation is that the integrand in the inner simulation may enjoy low effective dimension, which is friendly to QMC method \cite{Wang2006}. It is clear that the choice of the matrix $\Sigma$ has an impact on the effective dimension of the integrand.

Now we examine another risk factors structure, with $\Sigma$ satisfying $\tilde C=\Sigma\Sigma^T$, where $\tilde C_{ij}=0.3(d-|i-j|)/d$. \cref{fig:dimcompare} shows the decays of the variances of $Y_\ell$ and $\tilde Y_\ell$  with the new matrix $\Sigma$ for $d=32$ and $d=64$. For these cases, the strong convergence rate $\beta$ of the crude MLQMC methods is not much better than the MLMC method. But when it is combined with the dimension reduction technique GPCA, the strong convergence is improved dramatically, in which $\beta$ exceeds $1$. We thus conclude that the dimension of the integrand in the inner simulation has a significant impact on the performance of MLQMC methods. As the dimension gets higher, plain QMC methods may render a small $\eta<2$  in \cref{assm1}. When combined with dimension reduction methods, such as the GPCA method, the efficiency of QMC methods can be reclaimed, achieving a better strong convergence. This highlights the importance of using dimension reduction methods in nested MLQMC.

\section{Conclusion}\label{sec:conclusion}
In this paper, we have incorporated randomized QMC methods into MLMC to deal with financial risk estimation via nested simulation. We have proved that the new MLQMC estimator can achieve better complexity bound under some assumptions. At the meantime, we discussed the catastrophic coupling phenomenon caused by the character of RQMC points and discontinuity of the indicator function in the problem. Then we developed a new smoothed MLMC method and we proved that the smoothed MLMC method still reserves the advantages of MLMC without extra requirement. The smoothed method can also take advantage of antithetic sampling, which does not work for the original method. The superiority of the (smoothed) MLQMC methods over the original MLMC method has been empirically shown in numerical studies.

Further improvements for the MLQMC methods can be expected in the follow respects.
On the one hand, we only applied RQMC methods in the inner simulation. If RQMC methods are used not only in the inner sampling but also in the outer sampling, the computational cost may be further reduced. And the smooth coupling estimator is more suitable for RQMC methods and an improved efficiency can be expected.
On the other hand, we take a fixed number of inner samples in an identical level in this paper. If we are able to allocate the sample size according to the outer samples adaptively under the QMC scheme, like the way in the Monte Carlo scheme \cite{broa:2011,giles:2018}, a further improvement can be also expected.

\section*{Acknowledgments}
ZH would like to appreciate Prof. Micheal
B. Giles and Dr. Abdul-Lateef Haji-Ali  for useful discussions and comments at an early stage of this research.


\end{document}